\documentclass[10pt. oneside]{amsart} 
\usepackage[lmargin=1in,rmargin=1in,
bmargin=1in, tmargin=1in]{geometry}
\usepackage[dvipsnames]{xcolor}
\usepackage{amsmath, amssymb,tikz}
\usepackage{tikz-cd}
\usepackage{tikz}
\usepackage{todonotes}
\usepackage{mathrsfs}
\usepackage{extarrows}
\usepackage{graphicx}
\usepackage[breaklinks, pagebackref]{hyperref}
\usepackage{IEEEtrantools}
\usepackage{mathrsfs}
\usepackage[utf8]{inputenc}
\usepackage[english]{babel}
\usepackage{comment}
\usepackage{csquotes}
\usepackage[shortlabels]{enumitem}
\usepackage{mathrsfs}
\usepackage{stmaryrd}
\usepackage[shortlabels]{enumitem}
\usepackage{nicematrix}  
\usepackage{tikz-cd}
\usepackage{tikz}
\usetikzlibrary{decorations.pathreplacing,matrix,calc,positioning}
\usepackage{nicematrix}
\tikzset{ 
    table/.style={
        matrix of math nodes,
        row sep=-\pgflinewidth,
        column sep=-\pgflinewidth,
        nodes={rectangle,text width=3em,align=center},
        text depth=1.25ex,
        text height=2.5ex,
        nodes in empty cells,
        left delimiter=[,
        right delimiter={]},
        ampersand replacement=\&
    }
}
\usetikzlibrary{decorations.pathreplacing, calligraphy}

\hypersetup{
    colorlinks=true,
    linkcolor=blue,
    filecolor=red, 
    citecolor=red,          
    urlcolor=gray,
    pdftitle={distrelEis},
    pdfpagemode=FullScreen,
    }
\makeatletter
\newcommand*{\encircled}[1]{\relax\ifmmode\mathpalette\@encircled@math{#1}\else\@encircled{#1}\fi}
\newcommand*{\@encircled@math}[2]{\@encircled{$\m@th#1#2$}}
\newcommand*{\@encircled}[1]{%
  \tikz[baseline,anchor=base]{\node[draw,circle,outer sep=0pt,inner sep=.2ex] {#1};}}
\makeatother

\usepackage{pict2e,picture}

\makeatletter
\newcommand{\pnrelbar}{%
  \linethickness{\dimen2}%
  \sbox\z@{$\m@th\prec$}%
  \dimen@=1.1\ht\z@
  \begin{picture}(\dimen@,.4ex)
  \roundcap
  \put(0,.2ex){\line(1,0){\dimen@}}
  \put(\dimexpr 0.5\dimen@-.2ex\relax,0){\line(1,1){.4ex}}
  \end{picture}%
}

\newcommand{\precneq}{\mathrel{\vcenter{\hbox{\text{\prec@neq}}}}}
\newcommand{\prec@neq}{%
  \dimen2=\f@size\dimexpr.04pt\relax
  \oalign{%
    \noalign{\kern\dimexpr.2ex-.5\dimen2\relax}
    $\m@th\prec$\cr
    \noalign{\kern-.5\dimen2}
    \hidewidth\pnrelbar\hidewidth\cr
  }%
}
\makeatother

\DeclareFontEncoding{LS1}{}{}
\DeclareFontSubstitution{LS1}{stix}{m}{n}

\makeatletter
  \newcommand*\textmathversion{\csname textmv@\math@version\endcsname}
  \newcommand*\textmv@normal{m}
  \newcommand*\textmv@bold{b}
\makeatother
\DeclareFontEncoding{LS1}{}{}
\DeclareFontSubstitution{LS1}{stix}{m}{n}

\makeatletter
\usepackage{capt-of}

\DeclareFontEncoding{LS1}{}{}
\DeclareFontSubstitution{LS1}{stix}{m}{n}

\makeatletter

\DeclareFontEncoding{LS1}{}{}
\DeclareFontSubstitution{LS1}{stix}{m}{n}

\makeatletter

\DeclareFontEncoding{LS1}{}{}
\DeclareFontSubstitution{LS1}{stix}{m}{n}

\makeatletter

\usepackage{setspace}\setdisplayskipstretch{}
\usepackage{dynkin-diagrams}

\usepackage{amsmath,amsthm,amssymb}
\usepackage{colonequals}

\newcommand{\QQ}{\mathbb{Q}}

\newcommand{\ZZ}{\mathbb{Z}}

\newcounter{dummypart}



\newcommand{\CC}{\mathbb{C}}

\newcommand{\ch}{\mathrm{ch}}

\newcommand{\delT}{\mathbb{S}}   
    
\newcommand{\Sh}{\mathrm{Sh}}    
\newcommand{\Ab}{\mathbb{A}}   
    
\newcommand{\pr}{\mathrm{pr}}

\DeclareMathOperator{\supp}{supp}

\numberwithin{equation}{subsection}

\newtheorem{theorem}[equation]{Theorem}
\newtheorem{proposition}[equation]{Proposition}
\newtheorem{lemma}[equation]{Lemma}
\newtheorem{corollary}[equation]{Corollary}
\newtheorem*{corollary*}{Corollary}

\newtheorem{theoremx}{Theorem}

\theoremstyle{definition}

\newtheorem{definition}[equation]{Definition}

\theoremstyle{remark}
\newtheorem{remark}[equation]{Remark}

\newtheorem{note*}[equation]{Note}

\tikzcdset{scale cd/.style={every label/.append style={scale=#1},
    cells={nodes={scale=#1}}}}

\usepackage{tikz-cd}
\usepackage{mathrsfs}

\DeclareFontFamily{U}{wncy}{}
\DeclareFontShape{U}{wncy}{m}{n}{<->wncyr10}{}
\DeclareSymbolFont{mcy}{U}{wncy}{m}{n}
\DeclareMathSymbol{\sha}{\mathord}{mcy}{"58}

\makeatletter
\renewcommand*\env@matrix[1][\arraystretch]{%
  \edef\arraystretch{#1}%
  \hskip -\arraycolsep
  \let\@ifnextchar\new@ifnextchar
  \array{*\c@MaxMatrixCols c}}
\makeatother

\newcommand{\Gb}{\mathbf{G}}

\newcommand{\Addresses}{{
  \bigskip
\footnotesize
  (Shah) \textsc{Department of Mathematics,
University of California
Santa Barbara, CA 93106-3080}\par\nopagebreak
  \textit{E-mail address}: \texttt{swshah@ucsb.edu}
}}

\usepackage{mathtools,bm}

\DeclarePairedDelimiterX\Set[1]\{\}{%
  #1%
}

\usepackage{stmaryrd}

\newcommand{\GG}{\mathbb{G}}

\newcommand{\GL}{\mathrm{GL}}

\newcommand{\Et}{\mathrm{\acute{E}t}}

\newcommand{\et}{\mathrm{\acute{e}t}}

\newcommand{\RR}{\mathbb{R}}

\usepackage{setspace}
\usepackage[all]{xy}
\usepackage{textcomp}
\usepackage{amsbsy}
\usepackage{datetime}
\renewcommand{\dateseparator}{-}
\renewcommand{\today}{\the\year \dateseparator \twodigit\month
\dateseparator \twodigit\day}

\title{On distribution relations of polylogarithmic Eisenstein classes
}
\author{Syed Waqar Ali    Shah}    
\date{}

\begin{document} 
\begin{abstract}   We show that    for Siegel modular varieties of arbitrary genus, the  natural  distribution relations satisfied by certain integral Eisenstein  cohomology  classes   defined    by Kings admit an adelic refinement. This generalizes the classical relations for Siegel units on modular  curves.      
\end{abstract}    

\maketitle
\tableofcontents
\section{Introduction}  

\label{introsec}  

Motivic cohomology classes such as  Beilinson's  Eisenstein symbols  have several   important arithmetic applications. Kato  \cite{kkato}     used these symbols       to construct Euler systems for Galois representations attached to 
newforms and obtained spectacular results towards $p$-adic Birch-Swinnerton Dyer and Iwasawa main conjectures in these settings. The construction of Kato's Euler system makes essential   use 
of the so-called \emph{distribution relations} satisfied by these elements.  These  relations  describe the behaviour of Eisenstein symbols  under pushforward, pullback and conjugation morphisms between modular curves.  
Colmez  \cite{Colmez} later gave an adelic reformulation of these relations in  the analogous setting of 
modular forms. The adelic version is more useful since, e.g.,      it aids the construction of Euler systems via representation theoretic methods.

The construction of Beilinson's symbols can be generalized to other Shimura varieties by means of  polylogarithms on  abelian schemes. The resulting classes are  again of  motivic origin and 
referred to as  Eisenstein classes by analogy.  Using an Iwasawa theoretic  approach, Kings \cite{Kings15} showed that the $p$-adic  \'{e}tale realization of these classes satisfies   a $p$-adic interpolation property  in  varying  weights. In particular,  they enjoy certain integrality properties.

The modest  purpose of this note is to verify that the analogous adelic distribution relations  hold   integrally  for polylogarithmic Eisenstein classes constructed in the cohomology Siegel modular varieties of arbitrary  genus. For genus greater than one,  Colmez's argument does not immediately transfer to Kings' setting  owing partly to the failure of Galois descent in cohomology and requires solving a non-trivial     lifting problem.  The adelic relations play  a pivotal role in the     construction of Euler systems via pushforwards of test vectors     in the cohomology of more general Shimura varieties e.g.,  see  \cite{LSZ}. The formulation presented here is also needed in two forthcoming works \cite{Siegel2} and \cite{EulerGU22}.  It is our hope that the adelic extension of Kings' theory presented here for higher genus will find similar applications.  
\subsection{Main result}    
Let $( V_{\ZZ}, \psi) $ denote the standard symplectic $ \ZZ$-module of rank $ 2n $
(see  \S \ref{Siegelreview})            and 
$ \Gb  :  = \mathrm{GSp}_{2n}(V_{\ZZ}, \psi) $ denote  the $ \ZZ$-group scheme of  automorphisms of $ V_{\ZZ} $ that preserve $ \psi $ up to a scalar. For any ring $ R $, we denote $ V_{R} := V_{\ZZ} \otimes_{\ZZ} R  $.  Let $ p $ be a rational  prime and  $ c > 1 $ an integer with $ (c,p) = 1   $. We denote $ \ZZ_{cp}  :  = \prod_{\ell  |    cp} \ZZ_{\ell} $,   $   \Ab    _{f}^{cp}  :  = \Ab_{f}/\ZZ_{cp}  $ the group of finite rational  adeles away from primes dividing $ cp $  and  $ G := \Gb ( \Ab_{f}^{cp} )  \times  \Gb(\ZZ_{cp} ) $.  Let $ \mathcal{S}^{cp} $ denote the $ \ZZ_{p}$-module of all locally constant compactly supported  $ \ZZ_{p}$-valued   functions on $V_{\Ab_{f}} \setminus \left \{ 0 \right \} $ of the form $  \phi_{cp} \otimes \phi^{cp} $  where $ \phi_{cp} $ is the characteristic function of $ V_{\ZZ_{cp}} $ and $ \phi^{cp} $ is a function on $ V_{\Ab_{f}^{cp}} \setminus  \left \{ 0  \right  \}     $.  Then  $ \mathcal{S}^{cp} $ is a  smooth $ G$-representation. For $ K \subset G $ a subgroup, we denote by $ \mathcal{S}^{cp}(K)   \subset  \mathcal{S}^{cp}  $ the submodule of $ K $-invariants. For $ N \geq 1 $, let $ K_{N} \subset G $ denote the principal congruence subgroup of level $ N $, i.e., the subgroup of $ \Gb(\widehat{\ZZ}) $ that acts trivially on $ V_{\ZZ} / N V_{\ZZ} $. For each neat  compact open subgroup $ K \subset G $, let $ \mathrm{Sh}(K) $ denote the Siegel modular variety over $ \QQ $ of level $ K $ and $ \mathcal{A}_{K} \to \mathrm{Sh}(K) $ the universal abelian scheme.  Let $  \mathscr{H}_{\ZZ_{p}} = \mathscr{H}_{\mathcal{A_{K}}, \ZZ_{p}}  $ denote  the  $ \ZZ_{p}$-sheaf on $ \mathrm{Sh}(K) $ of $p$-adic Tate modules of $ \mathcal{A}_{K} $ and $ \mathscr{H}_{\QQ_{p}} $ the corresponding $ \QQ_{p}$-sheaf.  For $ k \geq 0 $,  let  $  \Gamma^{k} (\mathscr{H}_{\ZZ_{p}} ) $ (resp., 
 $  \mathrm{Sym}^{k} (  \mathscr{H}_{ \QQ_{p}} )) $ denote  the   $ k $-th divided power (resp.,  $k$-th 
 symmetric  power) sheaf.      For each torsion section $ t : \mathrm{Sh}(K) \to \mathcal{A}_{K} \setminus \mathcal{A}[c] $, Kings \cite{Kings15} has constructed  a \emph{$p$-adic \'{e}tale
 Eisenstein class}    
 $$ {}_{c} \mathrm{Eis}^{k} _ { \QQ_{p}} 
  (t) 
  \in  \mathrm{H}^{2n-1}_{\et}  \big ( 
 \mathrm{Sh}(K) ,  \mathrm{Sym}^{k} (  \mathscr{H}_{\QQ_{p}})   (  n )      \big   )    $$
 in the continuous  \'{e}tale   cohomology \cite{Jannsen1988} of $ \mathrm{Sh}(K)$.  For $K \subset G $ a neat compact open subgroup, let $ \mathcal{E}^{k}(K) $ denote the $ \ZZ_{p}$-submodule of $   \mathrm{H}^{2n-1}_{\et}     \big(\mathrm{Sh}(K),\mathrm{Sym}^{k}(\mathscr{H}_{\QQ_{p}})(n)\big) $  given by the image of  cohomology with coefficients in $ \Gamma_{k}(\mathscr{H}_{\ZZ_{p}})  ( n )     $.   The main result of \cite{Kings15}  implies that $ N^{k}  {}_{c}\mathrm{Eis}^{k}_{\QQ_{p}}  ( t )  \in \mathcal{E}^{k} (K ) $ if $ t $ is $N$-torsion for $ N $  satisfying $ (N,c) = 1 $.  For $ N \geq 3 $ and $ v \in V_{\widehat{\ZZ} }  \setminus  N  V_{  \widehat{\ZZ}   }     $, let $ t_{v,N} : \mathrm{Sh}(K_{N}) \to \mathcal{A}_{K_{N}} $ denote the section  corresponding  to     $ v + N V_{\widehat{\ZZ}}  \in    V_{\ZZ}/ N V_{\ZZ} $   under  the   
 universal level $ N $ structure on $ \mathcal{A}_{K_{N}}$ and $ \xi_{v,N}  \in  \mathcal{S}^{cp}(K_{N})  $ denote the  characteristic  function of $ v + N V_{\widehat{\ZZ}}    $.  Finally,      let $ \Upsilon $ be the collection of all compact open subgroups of $ G $  which are $G$-conjugate to a subgroup of $K_{N}$ for some  $ N \geq 3 $ satisfying $ (N,cp) = 1 $ and which are of  the form $ \mathbf{G}    (\ZZ_{cp}  )L $ for some   $ L \subset \Gb(\Ab_{f}^{cp}) $.    
 \begin{theoremx}[Theorem \ref{mainpararesult}]  \label{mainteo}    There exist a unique 
 collection of $ \ZZ_{p}$-module homomorphisms  $$ \varphi^{k}(K) : \mathcal{S}^{cp}(K)  \to  \mathcal{E}^{k}(K) $$     
indexed by   $ K \in \Upsilon $ satisfying the following conditions:
 \begin{itemize}   
  \setlength\itemsep{0.2em}
     \item  for each $ N \geq 3 $  prime to $ cp $ and $ v \in V_{\widehat{\ZZ}}  \setminus N  V_{\widehat{\ZZ} }  $,  $ \varphi^{k}(K_{N}) ( \xi_{v,N} )  =  N ^ { k }    {}_{c}   \mathrm{Eis}^{k}_{\QQ_{p}}(t_{v,N}) $,   
     \item for each $ K , L \in \Upsilon $ satisfying $ L \subset K $, we have  commutative diagrams
     \begin{center}
     \begin{tikzcd}[sep =  large, column sep = 5  
  em]   
     \mathcal{S}^{cp}(L )   \arrow[r, "\varphi^{k}(L)"]   \arrow[d, "{\pr_{*}}"']  &   \arrow[d, "{\pr_{*}}"]  \mathcal{E}^{k}(L)   &      \mathcal{S}^{cp}(L)   \arrow[r ,  " \varphi^{k}(L)"]     &    \mathcal{E}^{k}(L)   \\
     \mathcal{S}^{cp}(K)   \arrow[r,  "\varphi^{k}(K)"  ] &     \mathcal{E}^{k}(K)  & \mathcal{S}^{cp}(K)  \arrow[r  ,   "  \varphi^{k}(K)  " ]  \arrow[u, "{\pr^{*}}", hookrightarrow ]       
    &   \mathcal{E}^{k}(K)     \arrow[u,  "{\pr^{*}}"  '  ,  hookrightarrow ]    
     \end{tikzcd}
     \end{center}    
     where $ \pr_{*} $ and $ \pr^{*} $ denote respectively trace and inclusion  maps, 
     \item for each $ K \in  \Upsilon $ and $ g \in  G  $, we  have  a  commutative   diagram   
     \begin{center}   
     \begin{tikzcd}   [sep =  large,  column sep = 5 em]       \mathcal{S}^{cp}(K)  \arrow[r, "\varphi^{k}(K)"]   \arrow[d, "{[g]^{*}}"']    &   \mathcal{E}^{k}(K)   \arrow[d, "{[g]^{*}}"]  \\   
     \mathcal{S}^{cp}(gKg^{-1})   \arrow[r, "\varphi^{k}(gKg^{-1})"]   &    \mathcal{E}^{k}(gKg^{-1})
     \end{tikzcd} 
     \end{center}     
where $ [g]^{*}  $  denotes the (contravariant) conjugation  isomorphisms. 
 \end{itemize}
\end{theoremx}     

We refer to our result as an \emph{integral  parametrization} of Eisenstein classes by Schwartz spaces. For $ n = 1 $ and $ k = 0 $, our  result recovers the distribution relations for Kummer images of Siegel units proved in   \cite[\S 1-2]{kkato}.    We remark  that the classes $ {}_{c}\mathrm{Eis}_{\QQ_{p}}^{k}(t_{v,N}) $ for $ k \geq 1 $ in this case are closely related  to     the     Soul\'{e}  twisting construction applied to the classes for $ k = 0 $ \cite[\S 4.7]{Kings14}. For $ n \geq  2 $,  Lemma  \cite{LemmaResidue} has established  that these classes are not all zero for certain weights  $ k  $.       

The   core  ideas that go into the proof of Theorem \ref{mainteo}   are derived from \cite{kkato} and \cite{Colmez}.  However, we  must   carefully  address some complications not encountered in these works.  One of the issues is in defining the maps $\varphi^{k}(K) $ at, say, principal levels $K = K_{N} $   for functions that are not supported on $ V_{\widehat{\ZZ}} \setminus    \left     \{ 0 \right \} $ and constant modulo $ N V_{\widehat{\ZZ} } $. More precisely, there is no  obvious way to attach a linear combination of torsion sections for such functions. In the setting of \cite[\S 1] {Colmez}, the passage from integral to rational adeles is made by defining the action of the adelic group  in  the limit, using  which   maps at finite level can be recovered   by     taking invariants. But since the $K$-invariants of the $G$-representation $ \varinjlim_{ L } \mathcal{E}^{k} ( L )    $ are not necessarily equal to $ \mathcal{E}^{k}(K) $,  one cannot define $ \varphi^{k}(K) $ in this manner without   potentially  violating integrality,  a   crucial  requirement in the context of  Euler systems. So one has  to  construct for  each   $K$-invariant Schwartz function a class in $ \mathcal{E}^{k}(K) $ which  lifts the corresponding  class in  the  limit.      We show,  among  other   things,   that       these    lifts can indeed be constructed compatibly  for all   levels $ K \in  \Upsilon $  using   Hecke correspondences, Galois descent for torsion sections and  some elementary  topological  properties of the action of $ G  $  on $V_{\Ab_{f} } $.    

\subsection{Acknowledgements}The author is grateful to Guido Kings for answering several questions, to Dori Bejleri and Ananth Shankar for clarifying some details in \S \ref{Siegelsec}, and to Andrew Graham for several useful discussions. The author extends gratitude to the anonymous referee for their feedback on an earlier version of this note  and to David Loeffler for sharing his thoughts on an alternative strategy for establishing these relations in the case of modular curves.
\section{$p$-adic polylogarithms} 
\newcommand{\Log}{\mathcal{L}\mathrm{og}} 
\label{eisensteinsec}
In this section, we recall  the construction and $p$-adic interpolation of   Eisenstein  classes via polylogarithms following  \cite{HuberKings} and \cite{Kings15}. The  main  purpose is to  establish  some basic ``distribution relations"   (\S \ref{distributioneisenstein}) that describe the effect of  isogenies and base change on these  classes.  Except for \S\ref{distributioneisenstein}, the content  is  based on the original works and we skip most of the proofs within  these subsections.   

Throughout, $ \mathbf{Sch} $  denotes     the category of finite type separated schemes over a fixed Noetherian regular scheme of dimension at most $ 1 $.  For any $ X \in \mathbf{Sch} $, we let $ \mathbf{Sh}(X_{\et}) $ denote the category of     \'{e}tale sheaves of abelian groups on the small \'{e}tale site $ X_{\et} $ of $ X $ and  $ \mathbf{Sh}(X_{\et})^{\mathbb{N}}  $ the category of inverse systems on $ \mathbf{Sh}(X_{\et}) $. The $i$-th right derived functors of inverse limits of global sections of $ \mathscr{F}  = ( \mathscr{F}_{n})  _ { n \geq  1  }      \in \mathbf{Sh}(X_{\et}) ^ { \mathbb{  N } }  $ is denoted $ \mathrm{H}^{i}_{\et}(X, \mathscr{F}) $ and referred to as continuous \'{e}tale cohomology \cite{Jannsen1988}.    For $ p $ a rational prime  invertible on  $ X$,  we let $  \Et(X)_{\ZZ_{p}}  $ denote the abelian category of (constructible) $ \ZZ_{p} $-sheaves and its isogeny category of $ \QQ_{p} $-sheaves  by $ \Et(X)_{\QQ_{p}} $.  For $ \Lambda \in \left \{ \ZZ_{p}, \QQ_{p} \right\} $, we let $ \mathscr{D}(X )_{   \Lambda  }  $ denote the bounded ``derived"   category of $ \Et(X)_{ \Lambda } $ in the sense of \cite[Theorem 6.3]{Ekedahl}\footnote{See  \cite[\S 0]{HuberKings99} for a short explanation how this differs from an  ordinary derived category. See also \cite[Appendix A.1]{Anticyclo}.}.  There is a   full six functor formalism on these  and for any $ \mathscr{F} \in  \Et(X)_{\Lambda} $, we have   $ \mathrm{Hom}_{\mathscr{D}(X)_{\Lambda}}(\Lambda, \mathscr{F}[i]) = \mathrm{H}^{i}_{\et}(X , \mathscr{F} ) $ for all $ i \geq 0 $ 
by   \cite[Lemma 4.1]{Huber}.

\subsection{Purity for unipotent sheaves}   
\label{unipotentsection}    
Let $ \pi :X \to S $  be a  separated morphism of finite type   in   $ \mathbf{Sch} $.    A $ \ZZ_{p} $-sheaf $ \mathscr{F}   \in   \mathrm{\acute{E}t} ( X )_{\ZZ_{p}} $ is said to be $S$-\emph{unipotent of length $ k $}  if there exists a decreasing filtration $ \mathscr{F} = \mathscr{F}^{0}  \supset   \mathscr{F}^{1}  \supset \ldots \supset  \mathscr{F}^{n}  \supset \mathscr{F}^{n+1} = 0   $  such that  the $ \mathscr{F}^{i}/ \mathscr{F} ^{i+1} $ are isomorphic to $ \pi ^ { * }  \mathscr{G}  ^  { i   }  $   for $ \ZZ_{p} $-sheaves    $ \mathscr{G}  ^ { i  }  \in   \mathrm{\acute{Et}}(S)_{\ZZ_{p}}    $. We  refer to $ \mathscr{G}^{i} $ as the $i$-th graded piece of $ \mathscr{F} $.   We can similarly define unipotence of  $ \Lambda $-sheaves for $ \Lambda = \QQ_{p} $ or  $ \Lambda = \ZZ / p^{r}  \ZZ    $ (i.e., \'{e}tale $p^{r}$-torsion sheaves). 
\begin{lemma}
\label{unipotentshrieks}  Let $ \Lambda \in \left \{ \ZZ / p^{r} \ZZ , \ZZ_{p} , \QQ_{p} \right \} $.   Suppose $ \pi_{i} :  X_{i} \to S $ for $ i = 1 , 2 $ be morphisms as above  such that $ \pi_{i} $ is smooth of relative dimension $ d_{i} $. Let $  f : X _{1} \to X_{2 }$ be any  $ S $-morphism.   Then for any $S$-unipotent $ \Lambda$-sheaf $ \mathscr{F} $ (of some  finite length), $ f^{!}  \mathscr{F}   \simeq   f^{*}   \mathscr{F}  (d_{1} - d_{2} ) [2d_{1} - 2d_{2}   ]  $ functorially in $ \mathscr{F} $. 
\end{lemma}

\subsection{The $\QQ_{p}$-logarithm} 
\label{Qppolylogsection}    
Let   $ \pi : A \to S \in  \mathbf{Sch}  $ denote an abelian scheme     of relative dimension $ d $, i.e.,  $ A $ is a  group scheme and $ \pi $ is a smooth proper morphism with connected geometric fibers of dimension $ d $. The unit section is denoted by $ e : S \to  A $.   Let $ p $ be a  prime  invertible on $ S  $.      The \emph{$ p $-adic Tate  module} of $ \pi $ is defined to be first relative homology 
\begin{equation}   \label{tate}      \mathscr{H}_{\ZZ_{p}}   :  =  \mathscr{H}om_{S}(R^{1}\pi_{*}\ZZ_{p} ,  \ZZ_{p}) =   R^{2d-1}  \pi_{* } \ZZ_{p}(d)  
\in  \mathrm{\acute{E}t}(S)_{\ZZ_{p}}
\end{equation} of $  A $ with respect to $ S $. It is a lisse $ \ZZ_{p}$-sheaf and fiberwise equals the Tate module. We  let $ \mathscr{H}_{\QQ_{p}}  :    = \mathscr{H} _ { \ZZ_{p}   }     \otimes \QQ_{p} \in  \mathrm{\acute{E}t}(S)_{\QQ_{p} } $ denote  the corresponding $ \QQ_{p}$-sheaf. For $ r \geq 1 $, we similarly define $  \mathscr{H}_{\Lambda_{r}}    $ where  $ \Lambda _{ r }    :   = \ZZ / p ^{r} \ZZ  $.    Then  \begin{equation}   \label{HrisApr}    A [p^{r}]  \simeq  \mathscr{H}_{\Lambda_{r}} \simeq  \mathscr{H}  _ { \ZZ_{p }   }     \otimes _ { \ZZ_{p}   }    \ZZ /p^{r} \ZZ. 
\end{equation} where $ A[p^{r}] $ on the left denotes    the associated  representable  sheaf of $ p^{r}$ torsion in $ A $.  In what follows, we denote $ \mathscr{H}_{\ZZ_{p}} $ simply by $ \mathscr{H} $ if no   confusion can arise.

The low term exact sequence  for  the Leray spectral sequence associated with    $ \mathscr{H}om_{S}(\ZZ_{p}, -) \circ \pi_{*} $ evaluated at  $ \pi^{*}  \mathscr{H} $ gives   
\begin{align*} 0  \to    \mathrm{Ext}^{1}_{S} ( \ZZ_{p} ,  \mathscr{H} )  \xrightarrow { \pi ^{*} }     \mathrm{Ext}^{1}_{A 
} (   \ZZ_{p}  ,   \pi ^  { * } \mathscr{H}    )   &  \to \mathrm{Hom}_{S}(\ZZ_{p} ,  R^{1} \pi_{*} \pi ^{*}  \mathscr{H }     ) \\&  \to  \mathrm{Ext}_{S}^{2}(\ZZ_{p}, \mathscr{H})   \xrightarrow{ \pi^{*}}  \mathrm{Ext}_{A
}^{2}  (\ZZ_{p} , \pi^{*}  \mathscr{H} ) 
\end{align*}
The  maps $ \pi^{*} $ are  necessarily    injective as $ e^{*} \circ  \pi^{*} = ( \pi \circ e )^{*} = \mathrm{id}     $   and   therefore   the  morphism to the second line above  is    $ 0 $. By the projection formula, $ R^{1} \pi_{*}  \pi^{*}  \mathscr{H}  \simeq R^{1} \pi_{*}    \ZZ_{p}  \otimes \mathscr{H} $. Since $ R ^ {1 } \pi_{*} \ZZ_{p} \simeq \mathscr{H}om_{S}(\mathscr{H} , \ZZ_{p} )  =  :  \mathscr{H} ^ { \vee }   $ 
and   $ \mathscr{H}^{\vee}  \otimes   \mathscr{H}  \simeq \mathscr{H}om  _{S}       (\mathscr{H}, \mathscr{H} )$, we get   a split short  exact  sequence 
\begin{equation}   \label{exactsequencepoly}       0  \to  \mathrm{Ext}^{1} _ {S  } (    \ZZ_  { p }  ,      \mathscr{H} )   \to   \mathrm{Ext}^{1}_{A}  ( \ZZ_{p} ,  \pi^{*}   \mathscr{H}  )      \to  \mathrm{Hom}_{S}( \mathscr{H}  ,  \mathscr{H} )  \to  0 .   
\end{equation} 
with left splitting given by $ e ^ { * }    $.  Let $ \phi : \mathrm{Hom}_{S} ( \mathscr{H} , \mathscr{H})  \to  \mathrm{Ext}^{1}(\ZZ_{p} ,  \pi^{*}  \mathscr{H} ) $ denote the unique right splitting   satisfying    $ e^{*} \circ \phi  = 0  $.  
\begin{definition}  \label{firstlogarithm}    The \emph{first logarithm sheaf} is defined to be the pair $ ( \Log^{(1)}  _  {   \ZZ _ { p }  }     , \mathbf{1}^{(1)}    
)    $ where $ \Log^{(1)} _{\ZZ_{p}} =  \Log^{(1)}_{A, \ZZ_{p}}  \in  \mathrm{\acute{E}t}(A) _{ \ZZ_{p}}   $ is such that $ \phi(\mathrm{id}_{\mathscr{H}})  \in \mathrm{Ext}^{1}_{A}(\ZZ_{p}, \pi^{*} \mathscr{H}) $ is represented by 
\begin{equation}   \label{Logsequence}      0  \to  \pi^{*}  \mathscr{H}  \to  \mathcal{L} og ^ { ( 1 ) } _  { \ZZ_{p }  }     \xrightarrow { \delta }  \ZZ_{p}  \to  0
\end{equation}
and $ \mathbf{1}^{(1)} : \ZZ_{p} \to  e^{*}  \Log^{(1)}_{\ZZ_{p}}    $ is a  fixed right  splitting of the pullback of  (\ref{Logsequence}) under  identity\footnote{this pullback is necessarily split since $ e^{*} \circ \phi  = 0 $}. 
The pair $   \big   ( \Log_{\ZZ_{p}}^{(1)} , \mathbf{1}^{(1)}  \big   )  $ is then unique up to a unique  isomorphism.     We denote by $ \Log^{(1)}_{\QQ_{p}} $  the associated $ \QQ_{p} $-sheaf.
\end{definition} 
By  definition,    $ \Log^{(1)}  _ { \ZZ_ { p }    }       $  is  $S$-unipotent of length  one  
(see    \S \ref{unipotentsection}). One defines $ \Log^{(1)}_{\Lambda_{r}} $  for $ r \geq 1 $  in the same way as $ \mathscr{H}_{\Lambda_{r}} $. 
Then $ \Log ^ { (1  ) }  _ { \Lambda _ { r }  }  =  \Log  ^ { ( 1 )  }  _  { \ZZ_{p} }    \otimes  \Lambda_{r}  $ and $  \Log^{(1)}  _ { \ZZ_{p}}   =    (  \mathcal{L} og ^{(1)}  _ { \Lambda _{r} }  ) _{ r \geq  1 }  $. 
\begin{definition} For $ k \geq 1 $, the \emph{$ k $-th $ \QQ_{p}$-logarithm sheaf}  is the pair $  ( \Log^{(k)}_{\QQ_{p}} ,  \mathbf{1}^{(1)} ) $ where $  \Log^{(k)} _ { \QQ_  { p } }   : =    \mathrm{Sym}^{  k  }   \big (  \Log ^{(1)}  _ { \QQ_{p} }  \big )     \in   \mathrm{\acute{E}t}(A)_{\QQ_{p}}   $ and $$ \mathbf{1}^{(k)} : =  \frac{1}{k!}   \mathrm{Sym}^{k} ( \mathbf{1} ^{(1)}    )     :  \QQ_{p} \to   e^{*}      \Log   ^ {   (  k  )   }   _ {  \QQ_{p}  }      $$  is the  splitting map 
  induced by $ \mathbf{1}^{(1)} $ on the symmetric power.      
 \end{definition} 
The $ \QQ_{p} $-logarithms for $ k \geq 1 $ and their canonical splittings fit into an inverse system as follows.  Let $ \beta  :    =  \delta \oplus  \mathrm{id}  :  \Log_{\QQ_{p}}^{(1)}  \to  \QQ_{p}  \oplus  \Log ^{(1)} _  { \QQ_{p}} $ denote the diagonal map given by the sum of  the     projection $ \delta  $ in  (\ref{Logsequence}) and  identity. For 
$ k  \geq   2   $,   define   \emph{transition  maps}    $       u ^ { k  }     : \Log^{(k)}   _ {  \QQ_  { p }   }      \to     \Log^{(k-1)}   _   {  \QQ_ { p }   }          $ via     
\begin{align*} \Log^{(k)} _ { \QQ_{p}}     =   \mathrm{Sym} ^ {k }  (  \Log^{(1)} _ { \QQ_{p} }    )
&  \xrightarrow{ \mathrm{Sym}^{k}  \beta }    \mathrm{Sym}  ^ { k }  ( \QQ_{ p }  \oplus  \Log^{(1)}   _ { \QQ_{p}   }        )  \\ 
 & \simeq      \bigoplus   \nolimits  _ { i + j   = k }   \mathrm{Sym} ^ { i }   ( \QQ_{p} )   \otimes      \mathrm{Sym } ^ { j } (  \Log ^  {  ( 1 ) }   _  {   \QQ_{p } }      )  \\
&   \xrightarrow {}   \mathrm{Sym} ^ { 1  } ( \QQ_{p}  )  \otimes   \mathrm{Sym}^{k-1} ( \Log  ^ { (  1 ) }   _ { \QQ_{p} }  )     \simeq   \Log^{(k-1)}   _   {  \QQ_ { p }   } 
\end{align*}
For $ k \geq 2 $, we claim that $   (  e^{*}      u ^ { k }    )   \circ  \mathbf{1}^{(k)} : \QQ _ { p } \to   e ^ { * }   \mathcal{L} og ^ { ( k   -  1    ) } _ { \QQ_{p}}     $ is equal to $ \mathbf{1}^{( k-1)} $.   
First note that   by  definition of $ \mathbf{1}^{(1)} $, we have   $ (e^{*} \beta) \circ \mathbf{1}^{(1)}  =   \mathrm{id}   \oplus    \mathbf{1}^{(1)}   $   as  maps 
   $    \QQ_{p} \to \QQ_{p} \oplus   e ^ { * }        \Log_{\QQ_{p}}^{(1)}  $. Therefore 
   \begin{align*}    \mathrm{Sym} ^{k} (   e^{*}  \beta )  \circ 
   \mathbf{1}^{k}    = \frac{1}{k!} 
 \mathrm{Sym} ^ { ( k ) } (\mathrm{id}     \oplus  \mathbf{1}^{(1)}  )   &  \simeq    \bigoplus_{i+j=k}  \frac{1}{k!} \binom{k}{i}  \left ( \mathrm{Sym}^{i} (  \mathrm{id} )   \otimes   \mathrm{Sym} ^ {   j } ( \mathbf{1} ^{(1)}   ) \right ) \\ & =    \bigoplus _ { i + j = k }  \frac{  1   } { i ! } \left (  \mathrm{Sym}^{i}  (  \mathrm{id  }  )   \otimes \mathbf{1}^{(j) } \right ).
\end{align*}
The  projection of the last sum above to the summand at $ i = 1 $   is equal to  
$   \mathbf{1}^{(k-1)}  $.
Since $ u^{k} $ is
obtained  from $ \mathrm{Sym}^{k} ( \beta )  $ by  post-composition with projection to 
$ i = 1 $ summand as well, the claim follows.  If we  set 
  $    \Log^{(0)}_{\QQ_{p}}  : = \QQ_{p} $, $ \mathbf{1}^{(0)}   : \QQ_{p} \to  e^{*}  \Log^{(0)}_{\QQ_{p}}   $  the   identity and  $ u^{1} :  = \delta $ 
  (\ref{Logsequence}),       we still  
have   $   ( e^{*}   u^{1} )    \circ \mathbf{1}^{(1)} =  \mathbf{1}^{(0)} $. By  construction,  we have for each $ k \geq 1  $ an  exact sequence  \begin{equation}   0  \to   \pi ^ { * }  \mathrm{Sym}^{k}   \mathscr{H}   _ {  \QQ _ { p    }  }   \xrightarrow{   }   \Log^{(k)}  _  {  \QQ _ { p } }  \xrightarrow{u  ^ { k  }   }   \Log^{(k-1)}  _ {  \QQ _ { p } }   \to  0   \label{unipotence}       \end{equation} whose pullback under   $ e $  splits, giving an identification 
$ e^{*}  \Log^{(k)}_{\QQ_{p}} \simeq    \prod  _ { i = 0 } ^ { k  }  \mathrm{Sym}^{i}   \mathscr{H}_{\QQ_{p}} $ such that 
 $ e^{*} u^{k} $ is identified with the projection map $ \prod_{i=0}^{k}  \mathrm{Sym}^{i}  \mathscr{H}_{\QQ_{p}} \to  \prod_{i=0}^{k-1} \mathrm{Sym}^{i}   \mathscr{H}_{\QQ_{p}}  $.  
 One sees by induction  that $  \Log^{(k)}  _ { \QQ_{p}    }    $  is  $S$-unipotent of  length   $  k    $ with graded pieces given by symmetric powers of $ \mathscr{H}_{\QQ_{p}} $.

\begin{definition} The  \emph{$ \QQ_{p} $-logarithm  prosheaf} $ (\Log_{\QQ_{p}} , \mathbf{1}) $ is the  pro-system    $  (  \Log^{(k)}  _   { \QQ_ { p } }      , \mathbf{1} ^{(k)} )_{k \geq 0 }  $ of $\QQ_{p}$-sheaves  whose     transitions   maps     are given by   $ u ^ { k  }    $  for  $ k   \geq   1    $.      
 \end{definition} 
The logarithm prosheaf     satisfies several  
important  properties. Below we record  the   ones
 needed  later   on.          
\begin{proposition}   [Pullback Compatibility]
\label{basechangelog}     Suppose that $ f : T \to S $ is a morphism  in   $   \mathbf{Sch}   $,     $ A_{T} : =  A \times _{S}  T $ denotes  the pullback of  $  A $ to $ T $ and $ f_{A} : A_{T} \to A $ denotes   the natural map.     
Then there are canonical isomorphisms  $$  f_{A}    ^{*} \big ( \Log_{A,\QQ_{p} } ^ { ( k) }   \big  )      \simeq \Log_{A_{T},\QQ_{p}}^{(k)}   $$ for all $ k\geq 0 $  such that $  f_{A}  ^{*} ( \mathbf{1}^{(k)}_{A} ) $ is identified with the splitting $ \mathbf{1}^{(k)}_{A_{T}} $.      These  isomorphisms  commute  with  transition  maps.      
\end{proposition}    

\begin{proposition}[Functoriality]   For   any 
 isogeny  $  \varphi  : A
 \to A '
 $ of abelian  schemes over $ S $,  there   are unique  isomorphisms   $$   \varphi_{\#} :   \mathcal{L}\mathrm{og}^{(k)}  _ {  A  
 ,  \QQ_{p}}   \to 
 \varphi^{!}  \Log ^{(k)} _ {  A ' 
 ,  \QQ_{p} }  \simeq  \varphi^{*}     \Log ^{(k)} _ { A  '    ,  \QQ_{p} }     $$   for all $ k \geq 0 $ such that $ \mathbf{1}^{(k)}_{A} $ is sent to $ \mathbf{1}^{(k)}_{A'}     $. These    isomorphisms commute with  transition maps and  their pullbacks under identity  induce $ \mathrm{Sym}^{k} \varphi_{*} $ on $k$-th graded pieces. 
     \label{functoriality}    
\end{proposition}

\begin{corollary}[Splitting principle]  Let $ \iota : D \to A      $ be a closed subscheme  contained in the kernel of an isogeny $ \varphi : A \to A ' $
and   $ \pi_{D} 
: D \to S $ denote  its    structure map. Then there exist isomorphisms  $$  \varrho_{D}^{k}  :  \iota  ^{*} \Log^{(k)} _ {A ,  \QQ_ { p } } 
\xrightarrow{   \sim } 
\prod   \nolimits            _ { i  =0 } ^ { k  }  \pi_{D}^{*} \mathrm{Sym} ^ { i } ( \mathscr{H}_{A    
,   \QQ_{p}} ) $$
for all $ k \geq 0 $  that   commute with transition maps and 
are independent of the isogeny $ \varphi $. 
\label{splittingprinciple}       
\end{corollary}   

\begin{proof}  Let $ e ' $ denote the identity section  of $ A ' $.   First assume that $ (\iota, D ) = (t, S) $ is a section of $ \ker \varphi $ over $ S $.
We define  $ \varrho^{k}_{t} $
as  the  composition    
$$  t ^{*} \mathcal{L}og^{(k)}_{A
, \QQ_{p}} \xrightarrow{t^{*} \varphi_{\#}} (e ') ^{*} \mathcal{L}og^{(k)}_{A',\QQ_{p}} 
 \xrightarrow {(e^{*} \varphi_{\#})^{-1}} e ^{*}  \mathcal{L}og_{A, \QQ_{p}} ^ {  (  k ) }     \simeq \prod\nolimits     _{i=0}^{k}  
   \mathrm{Sym}^{i} (   
 \mathscr{H}_{A , \QQ_{p}}  )   $$ where the last isomorphism is induced by (\ref{unipotence}) as above.   
If  $ \psi : A ' \to A '' $ 
 is any  isogeny over $ S $,
 the corresponding  isomorphism defined with 
 respect to $ \varphi ' :=  \psi  \circ  \varphi $ is easily seen to  coincide    with  $ \varrho^{k}_{t} $ using the cocycle condition  $ \varphi'_{\# } = ( \varphi^{*}  \psi )_{\#} \circ \varphi_{\#} $,   
  which holds  by 
 uniqueness  of   the  maps   involved.     As any two isogenies from $ A $ that annihilate $ t $ can be refined by a common isogeny (by quasi-compactness of $ S$),    $ \varrho^{k}_{t} $               does not depend on $ \varphi $.
 
 In  general,  let $ A_{D} := A \times _{S}  D  $  and  
$ t_{D}  : D \to A_{D} $ be the tautological  section obtained from $t $  by base change.  Then we define $ \varrho^{k}_{D} $ as $  \varrho^{k}_{t_{D}}$  after  identifying 
$    \iota^ { * }  \Log ^ { ( k ) } _ { A , \QQ_ { p } } \simeq  t _ { D } ^ { *  }   \mathcal{L}og_{A_{D}, \QQ_{p}}^{(k)}$ and $    \pi_{D}^{*} \mathscr{H}_{A, \QQ_{p}}    \simeq   
\mathscr{H}_{A_{D }  , \QQ_{p}}   $.                
\end{proof}

\begin{proposition}[Vanishing of cohomology]   \label{vanishingcohomology} There exist natural  isomorphisms $  R^{2d}   \pi_{*}     (  \Log^{(k)} _  { \QQ  _  { p }   }       )    \simeq \QQ_{p} ( - d )  $ for all $k  \geq 0 $ that commute with the transition maps.  For $  i = 0 , \ldots, 2d - 1 $,  the induced   maps $ R^{i}   u ^ { k }    :  R^{i }   \pi_{*}    \Log^{(k)} _ { \QQ_ { p }  }       \to     R ^ { i }  \pi_{*}   \Log ^ { ( k - 1 )  } _ { \QQ_{p} }       $ are zero for all $ k \geq 1 $.
In particular,    $$   
\varprojlim_{ k}   \mathrm{H}^{i}_{\et}  \big (A,   \Log_{\QQ_{p}} ^ {(k) }  (d)  \big   )
\simeq    \begin{cases}  0  &  \text{ if } i < 2d,  \\  \mathrm{H}^{0}_{\et} (S, \QQ_{p})   &  \text{ if }  i = 2d. 
\end{cases} $$  
\end{proposition}

\subsection{Cohomology classes}   \label{LogQpressection}     Fix  $ c > 1 $  an integer invertible on $ S $ and let  $ D :  =
A[c] $ be the group of $ c $-torsion points of $ A $. Then $ D $ is a finite \'{e}tale group scheme over $ S $. Let   
$ U =  U_{D} : =  A   \setminus D      $ denote   the complement of  $  D $ in  $ A $ and consider the diagram.  
\begin{equation}   \label{geometricsituation}   
\begin{tikzcd}[sep = large]      U  \arrow[dr, "\pi_{U}"']    \arrow[r , "j_{D}" ]  &  A   \arrow[d, "\pi"]      &  \arrow[l,   swap,    "\iota_{D}"]  D    \arrow[dl, "  \pi_{D}    "   ]  \\
& S & 
\end{tikzcd}
\end{equation}
where $  j = j_{D} $ and $  \iota  =  \iota_{D} $ are natural inclusions and $ \pi_{D} :  = \pi \circ \iota $, $ \pi_{U} : =  \pi \circ j $
are the structure  maps.  For any $ \mathscr{F}   \in  \mathscr{D}(A)_{\QQ_{p}}    $,     we have a   distinguished   triangle $   R  \iota_{*}  \iota^{!}  \mathscr{F}  \to   \mathscr{F} \to  R j_{*} j^{*}   \mathscr{F}     \to  R\iota_{*} \iota^{!} \mathscr{F} [1]  \, \in  \mathscr{D}(A)_{\QQ_{p}}  $ known as the
 \emph{localization} triangle.        
Applying $ R \pi_{*}  =  R \pi_{!}    $ to the localization triangle  with $ \mathscr{F} = \Log^{  ( k )    }   _  { \QQ_ { p } } (d) $ for a fixed $ k $,    we get a distinguished triangle
\begin{equation} \label{localizationtriangle} R \pi_{D, * } \iota^{!}  \Log ^ {  (k )  } _{\QQ_{p}}(d)  \to  R \pi_{*}  \Log^{  (k )  }_{\QQ_{p}}(d)   \to  R \pi_{U , * } j^{*}    \Log^{(k) }_{\QQ_{p}}(d)   \to  R \pi_{D , *}   \iota^{!}      \Log^{  (k )   } _{\QQ_{p}}(d) [1]  \in   \mathscr{D}(S)_{\QQ_{p}}          
\end{equation}
Using Lemma  \ref{unipotentshrieks}   and the fact that $ \Log^{ (k)}  _ { \QQ_{p }  }     $ is $S$-unipotent for each $ k $, we see that $ \iota^{!} \Log_{\QQ_{p}}^{  (k)  }  (d)  =  \iota^{*}  \Log  ^ {  (k) } _ { \QQ_ { p  } } [-2d] $ and therefore $ R \pi_{D, * }  \iota^{!} \Log_{\QQ_{p}}^{ (k)  } (d) =  R \pi_{D,  *  }    \iota^{*}  \Log_{\QQ_{p}}^{  ( k )  }  [-2d] $,  etc.    
Applying $ \mathrm{Hom}_{\mathscr{D}(X)_{\QQ_{p}}}  ( \QQ_{p}     ,  - ) $ and using the adjunctions $ \pi^{*} \dashv  R \pi_{*}$, etc., we obtain  a long exact sequence (\cite[Theorem II.1.3]{WeissauerKiehl}) 
\begin{align}  \notag     \cdots   \to \mathrm{H}^{2d - 1 } _{ \et}  \big ( A , \Log^{(k)}_{\QQ_{p}}(d) \big )   &  \to  \mathrm{H}^{ 2d - 1 }_{\et}   \big  (  U  ,    \Log^{ ( k )     }_{U , \QQ_{p}    } (d)   \big  )  \\   \label{longexactlogsequence}
&   \xrightarrow {}  \mathrm{H}^{ 0 } _{\et}  \big ( D , \iota^{*}  \Log ^{(k)}_{\QQ_{p}}  \big ) \to   \mathrm{H}^{2d}_{\et}   \big (  A , \Log^{ ( k )  }  _ { \QQ_{p}   }     (d)   \big   ) \to  \cdots 
\end{align}    
where $ \Log_{U,\QQ_{p}}  ^ { ( k )  }  $ denotes the pullback of $ \Log_{\QQ_{p}}  ^ { ( k )  }     $ to the open subset $ U  $.  By construction, these sequences  commute with maps induced by transition maps from $ k +1 $ to $ k $.       Abusing  notation,   we   denote the inverse limit over $ k $ of each of  the groups appearing in (\ref{longexactlogsequence}) by removing the superscript $(k)$.
Taking inverse limit of (\ref{longexactlogsequence}) over all $ k $, we obtain  an  exact sequence \begin{equation}   \label{logresiduesequence}         0  \to   \mathrm{H}    ^{2d-1}_{\et}   (  U  , \Log_{U,  \QQ_{p}}  ^ { }  ( d )   )   
\to  
\mathrm{H}^{0}_{\et}   (D , \iota^{*}  \mathcal{L}og_{\QQ_{p}}
 )  
\to   \mathrm{H}^{0}_{\et}(S,  \QQ_{p} )  
\end{equation}   
by Proposition  \ref{vanishingcohomology} and left exactness of inverse   limit.    The middle map appearing   in   (\ref{logresiduesequence}) is denoted $ \mathrm{res}_{U} $ and referred to as the \emph{residue map}.  The rightmost map in   (\ref{logresiduesequence})  is  induced by the composition of  augmentation $ \pi_{D,*} \iota^{*} \mathcal{L}og^{(k)}  _ { \QQ_{p}}   \to \pi_{D,*} \iota^{*} \QQ_{p} $ followed by   counit  adjunction   $ \pi_{D,*} \iota^{*} \QQ_{p} =   \pi_{D, !} \pi_{D}^{!}   \QQ_{p} \to \QQ_{p} $.  Via the identification of  Corollary  \ref{splittingprinciple}, the restriction of   the     rightmost map above to the $k=0$ component is the trace morphism $$\varepsilon_{D} :  \mathrm{H}_{\et}^{0}(D, \QQ_{p}) \to \mathrm{H}^{0}_{\et} (S,\QQ_{p}) . $$
Let $ \QQ_{p} [D] ^{0}  = \ker \varepsilon_{D} $. 

\begin{definition} \label{poldefi} The  \emph{polylogarithm  class with residue  $ \alpha \in \QQ_{p}[D] ^{0}$}  is the unique   cohomology     class  $$  \prescript {}{\alpha}  
{\mathrm{pol}_{\QQ_{p} }}\in\mathrm{H}^{2d-1}_{\et}\big(U,\Log_ {U,\QQ_{p} }(d)\big) $$ 
such that $ \mathrm{res}_{U} \big  ( 
{}_{\alpha}  { \mathrm{pol} } _ { \QQ_   {p}}    \big  )  =  \alpha $. For $ k \geq 0 $, we define $  {}_{\alpha} \mathrm{pol} ^{k}     _ { \QQ_{p} }   \in  \mathrm{H}^{2d-1}_{\et} \big (U, \Log^{(k)} _{U, \QQ_{p}}  (d)    \big   ) $ to be the image of $ {}_{\alpha} \mathrm{pol} _ {\QQ_{p}}$.      
\end{definition}    

\begin{remark}   \label{extremark}
We    have a similarly  defined  exact sequence (of groups of inverse limits) 
$$  0 \to  \mathrm{Ext}^{2d-1}_{U} \big  (
\QQ_{p} ,   \mathcal{L} og ^{   } _{ U ,  \QQ_{p}} (d)  \big  )  \xrightarrow {}  \mathrm{Hom}_{D} \big (\QQ_{p},\iota^{*}\mathcal{L}og^{}_{\QQ_{p}} \big )
\to  \mathrm{Hom}_{S} ({ \QQ_{p}} , \QQ_{p} )    $$ and a
class in the Ext group corresponding to $ \alpha \in \QQ_{p}[D]^{0} $. This is the perspective taken in
 \cite[\S 5.2]{HuberKings}.
 \end{remark}

Fix now an integer $ N >  1 $ that is invertible on $  S $ and such  that    $  (N , c ) =  1 $.  Let $ t : S \to A  $  be a non-zero $ N $-torsion section.  Then $ t $ factors via $ U $ as $ (N,c) = 1 $ and we will also denote $ t : S \to U $.         The unit  adjunction $  \mathrm{ id }  \to  R t _{ * } t ^{ * }  $ on $ U $ gives a  morphism \begin{align*} R \pi_{U , * }  \mathcal{L}og^{(k)}_{U,\QQ_{p}}(d) & \to  R\pi_{U, *}   R  t_ { * }  t ^ { * }  \mathcal{L}og  ^ {(k)}_{U,\QQ_{p}}(d)  \xrightarrow {\sim} R(\pi_{U} \circ t) _{*}t^{*} \mathcal{L}og^{(k)}_ {U,\QQ_{p}}(d) = t^{*}\mathcal{L}og^ {(k)}_ {U,\QQ_{p}}(d) 
\end{align*}
which in  turn  induces   the pullback 
\begin{equation}   \label{torsionpullback}     t ^ { * }  :     \mathrm{H} _ { \et } ^ {2d - 1} \big(U , \mathcal{L}og  ^ { ( k )  
}  _ {  U  ,     \QQ_    { p } }   (  d  )     \big       )   \to   \mathrm{H} ^ { 2d  - 1 } _ { \et  }   \big  (  S  ,    t ^ {* }  \mathcal{L} og  ^ { ( k ) }  _ {  U  ,     \QQ _ { p }  }   
 ( d)   \big )   \end{equation}     on cohomology.   By  Corollary
 \ref{splittingprinciple}, we have a map
\begin{equation}  
\mathrm{H} _ { \et } ^ { 2 d - 1 }   \big  (  S ,     t  ^  { *  }     \mathcal{L}og^{(k)}_{U,\QQ_{p}}(d) \big )   \xrightarrow [\sim ] { \varrho_{t}  ^{k}  } \mathrm{H}^{2d-1} _ { \et }   \big    ( S   ,   \prod   \nolimits _ {   i  =     0 
} ^{k}   \mathrm{Sym}^{i} (  \mathscr{H}_{\QQ_{p}}   )  ( d )   \big  )         \xrightarrow{\pr^{k}} 
\mathrm{H}^{2d-1} _ { \et }   \big    ( S  ,  \mathrm{Sym}^{k} ( \mathscr{H}_{\QQ_{p}} ) ( d )   \big  )   
\label{polylogtoeisenstein}    
\end{equation} 
where $ \pr^{k} $ is the projection on the $ k $-th symmetric power. 
\begin{definition}   \label{EisdefiQp}  Let $ \alpha \in \QQ_{p} [ D ]^{0} $, $ t  : S \to  U $ be a  non-zero  
  $ N $-torsion section as above and $ k \geq 0 $ be an integer. The $ k $-th \emph{rational   Eisenstein class}  $$ \prescript{}{\alpha}{  \mathrm{Eis} }^{k}  _ { \QQ_{p}}  ( t )   \in   \mathrm{H}^{2d-1} _ { \et }   \big    ( S  ,  \mathrm{Sym}^{k}   ( \mathscr{H}_{\QQ_{p}}   ) ( d )   \big  ) $$
\emph{along $ t $ with residue $ \alpha $}    is  the  image of  $ \prescript{}{\alpha}{\mathrm{pol}}_{\QQ_{p}}  ^{  k  } 
\in  \mathrm{H}^{2d-1}_{\et}\big (U, \Log_{U, \QQ_{p}} ^ { ( k  ) }  (d) \big ) $ under the composition  $ \pr^{k}  \circ   \varrho_{t} ^ { k } \circ  t^{*} $ of     (\ref{polylogtoeisenstein}) and  (\ref{torsionpullback}).     
 \end{definition}

There is a   special   choice  of  $ \alpha $   which       enjoys certain  compatibility   properties  and  which    will be used in \S \ref{parametrization}.         
 Let $ \pi_{D}^{*} :   \mathrm{  H  }    ^{0} ( S , \QQ_{p} ) 
\to  \mathrm{H}^{0}_{\et}( D, \QQ_{p}) $ denote the pullback map  and $   e_{*} :   \mathrm{H}^{0}_{\et}(S, \QQ_{p}) \to  \mathrm{H}^{0}_{\et}(D   , \QQ_{p})   $ the 
map  given by counit  adjunction  
$  \QQ_{p}  =  \pi_{D,  * }  e_{!} e^{!}   \QQ_  { p  }   \to   \pi_{D, * }  \QQ_{p}   $. Let   
\begin{equation}   \label{alphacchoice}   \alpha_{c}  =  \alpha_{A, c }  :  =   c ^  { 2 d }          e_{*} ( 1  ) - \pi_{D}^{*}(1   )   
\end{equation} 
where $ 1 \in  \mathrm{H}   ^{0}_{\et}(S, \QQ_{p}) $ denotes    the global section given  locally  by   $ 1 \in \QQ_{p} $.  Then 
$  \alpha_{c}  \in   \QQ_{p}  [ D    ]      ^ { 0  }     $  since   $ \pi_{D, * } ( \alpha_{c} ) =   c^{2d} ( \pi_{D} \circ  e)_{*}  (1)  -  \pi_{D, *} \pi_{D}^{*} (1) =  c^{2d}  - c^{2d} = 0   $.    
\begin{definition}   \label{cEisclass}      
We   denote by  $ {}_{c} \mathrm{pol}_{\QQ_{p}} $ (resp.,    $ {} _ { c} \mathrm{Eis}^{k}(t)_{\QQ_{p}} $)     the  polylogarithm (resp.,   Eisenstein)  class  with  residue $ \alpha_{c} $. \end{definition}  
\begin{remark}   See   \cite[\S 4]{Kings14} 
for  a precise relationship between these classes   in  the  elliptic  case    and Kato's Siegel units.  The general construction of the polylogarithm has its origins in the work of Beilinson and Levin \cite{BeilinsonLevin}, whose ideas were later placed in a much broader framework by Wildeshaus \cite{Wildeshaus} and  Kings \cite{KingsKtheory}, \cite{KingsNote}.  When $ S $ is the Siegel modular variety of genus two of a suitable level, Faltings  \cite{FaltingsEisenstein}  has also  constructed a ``potentially motivic" Eisenstein class in $\mathrm{H}^{3}_{\et}(S,  \QQ_{p}(3)) $, whereas the weight zero construction above yields a class in $ \mathrm{H}^{3}_{\et}(S,  \QQ_{p}(2))$.   This other class has recently been  used to construct a new Euler system in \cite{SkinnerVincentelli}.  
\end{remark}    

\subsection{Norm   compatibility}   \label{normcompsec}     We  maintain the notations of \S  \ref{LogQpressection}.   
Suppose now that   $ \pi ' :  A ' \to S $ 
is   another abelian scheme   with  unit  section $  e ' : S \to A ' $ and define $ D' $, $ U' $, etc., analogously  as in   diagram      (\ref{geometricsituation}).   For 
 notational clarity, we will  denote   $ \mathscr{F} =      
\mathcal{L}\mathrm{og}^{(k)}  _ {  A ,  \QQ_{p}} ( d )    $ and    $ \mathscr{G} =     
\mathcal{L}\mathrm{og}^{(k)}  _ {  A'  ,  \QQ_{p}} (d)   $  in this subsection.     

Let $ \varphi  : A \to A ' $ be a $ S$-isogeny and  
$ \varphi_{D}  : D \to D ' $ denote its restriction to $ D $. Set $ \tilde{D} : = \varphi^{-1}(D') $, $ \tilde{U} : =  \varphi^{-1}( U') \subset U  $ and denote by $   j_{\tilde{D}}  : \tilde{U}  
\to A $, 
$ \jmath :   \tilde{U}  \to     U $ the  inclusion maps.   The  unit  adjunction $  \mathrm{id}  \to  R \jmath _{*}   \jmath^{*}  $ gives a  restriction transformation $  {r} _{U, \tilde{U}} :  R j _{D ,* }  j_{D } ^{* }  \to  R   
 j_{D,*} R \jmath_{*} \jmath^{*}    j_{D}^{*}  \xrightarrow{\sim}      Rj_{\tilde{D},*}  j_{\tilde{D} }^{*} $.  Since $ \tilde{\varphi} : = \varphi_{\mid  \tilde{U}} :  \tilde{U} \to U  '    $ is the pullback of $ \varphi $  along $ j_{D'} $, we have  $ \tilde{  \varphi } ^{!} j_{D'}^{*}  \simeq   j_{\tilde{D}}^{*} \varphi^{!}  $ and  a base change isomorphism $ 
 \varphi^{!} R j_{D',*} 
  \xrightarrow{ \sim }  R  j_{\tilde{D},*}  \tilde {     \varphi}^{! }  $. 
Define
\begin{align*}  
\varphi^{\natural} : Rj_{D,*}  j_{D}^{*}   \mathscr{F}  
\xrightarrow{   \varphi_{\#}        }   
Rj_{D,*}  j_{D}^{*}   \varphi^{!}   \mathscr{G} 
  \xrightarrow{  r _{  U ,    \tilde{U}  }}    
 R j_{\tilde{D},*}  j_{\tilde{D}}^{*}  \varphi^{!}   \mathscr{G}  \xrightarrow{\sim}   R j_{\tilde{D},*}      \tilde{   \varphi}^{!} j_{ D   ' 
   }^{*}     
 \mathscr{G}   \xrightarrow{ \sim }   \varphi^{!}   R j _  { D ' ,
 * }   j _  {D   '    } ^  {  * }   \mathscr{G}  . 
\end{align*}
By \cite[Lemma 5.1.2]{HuberKings},   
there is a morphism of distinguished triangles in $  \mathscr{D}(A') _  { \QQ_{p}}  $     
\begin{center}   
\begin{tikzcd}  R \varphi_{*}  R\iota_{D,*}  \iota_{D}^{!}  
\mathscr{ F}     
\arrow[r]      \arrow[d]  & R \varphi_{*}    
\mathscr{F}  
\arrow[r]   \arrow[d]   &  R \varphi_{*}   R j_{D,*} j_{D}^{*}   
\mathscr{F}  
\arrow  [r ]  \arrow[d, "\varphi_{\natural}"] & 
     R \varphi_{*}  R  \iota_{D,*} \iota_{D}^{!}    
\mathscr{F}
[1]   \arrow[d]    \\
 R \iota_{D',*} \iota_{D'}^{!}   \mathscr{G}   \arrow[r]      &   \mathscr{G}   \arrow[r ]     &   R   j_{D', *} j_{D'}^{*}  \mathscr{G}   \arrow[r]   &      R\iota_{D',*} \iota_{D'}^{!}   \mathscr{G}  [ 1 ]    
\end{tikzcd}
\end{center}    
where   $ \varphi_{ \natural}  :  R  \varphi_{*}   Rj_{D,*}  j_{D}^{*}   \mathscr{F}   \to    R j _  { D ' ,  * }   j _  {D   '  } ^  {  * }   \mathscr{G}    $ denotes the mate of $ \varphi^{\natural} $  under the adjunction $ R  \varphi_{*} = R \varphi_{!}      \dashv  \varphi^{!} $. Applying $ R \pi'_{!} =  R \pi'_{*}   $,   we   obtain a diagram in  $ \mathscr{D}(S)_{\QQ_{p}} $ where the top row  is  (\ref{localizationtriangle}) and the bottom row is the version  defined for $ A '  $.   Repeating the steps of \S \ref{LogQpressection}, we obtain a \emph{norm map} $ \mathscr{N}_{\varphi}  $    from the short exact sequence  (\ref{logresiduesequence})  for $ A $ to that for $ A ' $.  On $ \mathrm{H}^{0}_{\et}(S, \QQ_{p}) $, $ \mathscr{N}_{\varphi} $  is just identity while the map $ \mathrm{H}^{0}_{\et}(D, \QQ_{p}) \to \mathrm{H}^{0}_{\et}(D', \QQ_{p}) $ is the trace $ \varphi_{D,*} $.  By uniqueness of polylogarithms   with  respect  to 
residues    and compatibility of adjunction morphisms, we obtain the following. 

\begin{proposition}[Norm  Compatibility]      \label{poltrace}     $  \mathscr{N}_{\varphi}  
(   { }  _ {  \alpha }   \mathrm{pol} _ {A , \QQ_{p}}) = {}_{\beta}   \mathrm{pol} _ { A' ,  \QQ_  { p }  } $ where  $  \beta =  \varphi_{D, *} (\alpha) \in \QQ_{p}[D  '    ]^{0} $. 
\end{proposition}   
We will   also     need the following result.        
\begin{lemma}   \label{alphacnorm}                     If $   \varphi $ has constant  degree and $ (\deg  \varphi , c ) = 1   $,   $ \varphi_{D,*}  ( \alpha_{c} )  
=  \alpha_{A', c} $.
\end{lemma} 

\begin{proof}
Since $ \varphi_{D,*} \circ  e_{*} (1) = (\varphi_{D} \circ e) _{*}(1) = e _{*}' (1)  $, we only need to show that  $  \varphi_{D,*} \circ \pi_{D}^{*}(1) = \pi_{D' } ^ {*  } ( 1) $. This can be established    \'{e}tale locally,  i.e.,    over  a finite \'{e}tale cover of $ S $ where $ D, D' $ become  constant group schemes on $ \Gamma   :     =  ( \ZZ /  c \ZZ )^{2d} $. In this case, $ \varphi_{D} : D \to D' $ is determined by an automorphism of $ \Gamma $ and $ \varphi_{D,*} $ identifies with the endomorphism on  $ \oplus_{\gamma} \mathrm{H}^{0}_{\et}(S, \QQ_{p}  ) $ given by identity maps between the permuted  components determined by the automorphism of $ \Gamma $.  As $  (1)_{\gamma \in \Gamma } $ is clearly preserved by such maps, the  claim  follows.    
\end{proof}

\subsection{Distribution  relations}   \label{distributioneisenstein} 
Fix  $    \alpha $, $ t $ and $ k $   as in Definition  \ref{EisdefiQp}  for all of this subsection.  For the next result,  we let  $ f : T \to S $ denote   a   fixed     morphism  in  $  \mathbf{Sch}  $.  Set    $ A_{T} : = A \times_{S}  T $,  $ D_{T} : = A_{T}[c] $ and $U_{T} : = A_{T} \setminus D_{T} $. We denote by $ f_{A} : A_{T} \to A $ the natural map and  by $ f_{D} : D_{T} \to D  $, $ f_{U} : U_{T} \to U  $ the restriction of $ f_{A} $ to $ D_{T} $, $ U_{T} $ respectively.  
Let $ e_{T} :  T \to A_{T} $ denote  the   identity  section  and  $  t_{T} :  T     \to U_{T} $ the tautological section induced by base changing    $ t $.   
\begin{equation}   \label{abelianschemepullback}     
\begin{tikzcd} [sep = large]      U_ {T }  \arrow[d,  "{\pi_{U_{T}}}", swap]   \arrow[r, "f_{U}"] &  U  \arrow[d, "\pi_{U}"']  \\
T   \arrow[u, swap,  dotted  , bend right, "t_{T}"] \arrow[r, "f"]  &  S    \arrow[u, bend right, swap, "t"]   
\end{tikzcd}  
\end{equation}    
Denote by
$$f^{*}_{k} : \mathrm{H}_{\et}^{2d-1} \big (  S , \mathrm{Sym}^{k}  (  \mathscr{H}_{A , \QQ_{p}} ) 
   ( d )  \big  )  \to   \mathrm{H}_{\et}^{2d-1} \big ( T , \mathrm{Sym}^{k} ( \mathscr{H}_{A_{T} , \QQ_{p}} )  (d)  \big ) , \quad  \quad   \quad 
f^{*}_{D} :   \mathrm{H}^{0}_{\et} ( D , \QQ_{p} )   \to    \mathrm{H}^{0}_{\et}  ( D_{T} , \QQ_{p} )   $$
the pullback maps induced by the   unit     adjunction $ \mathrm{id} \to R f_{*} f^{*} $, etc.  More generally, this  adjunction induces 
(via base change applied to (\ref{abelianschemepullback}) and Proposition \ref{basechangelog}) a morphism from the triangle  (\ref{localizationtriangle}) to $ R f_{*} $  applied to the corresponding triangle  for $ A_{T} $. Thus we get a  pullback  map  from the  sequence  (\ref{logresiduesequence})     to the corresponding sequence for $ A_{T} $.
\begin{lemma}   \label{Eispullbackcompatibility}
$ f^{*}_{k}    \big ({} _{ \alpha}  \mathrm{Eis}^{k}_{\QQ_{p}}(t)   \big) = {}_{  \beta   }  \mathrm{Eis}^{k}_{\QQ_{p}}  
( t_{T}  ) $ where $ \beta =  f^{*}_{D} ( \alpha ) $.  Moreover $ f^{*}_{D}(\alpha_{c} ) = \alpha_{A_{T} ,  c}  $.
\end{lemma}

\begin{proof}
Since $ \alpha \in  \QQ_{p} [D]^{0}  =  \ker  \epsilon_{D}  $,  the compatibility of (\ref{logresiduesequence})  along   pullbacks implies that  $ \beta \in \QQ_{p}[D_{T}]^{0} $. Let   $  {  }  _ { 
 \beta    }   \mathrm{pol}_{T, \QQ_{p}}  ^  { (  k  ) }
 $ denote the polylogarithm for $ A_{T} $ with residue  $ \beta $.
Then $  f_{U}^{*} ( {}_{\alpha}  \mathrm{pol}_{\QQ_{p}}  ^ {k} )    = {}_{ \beta } \mathrm{pol}_{T  ,    \QQ_{p} } ^ { k  } $ by uniqueness of these classes  with   respect  to  residues.  It is easily seen  from the proof of Corollary \ref{splittingprinciple}  
that $  \varrho^{k}_{t_{T}} $ is the pullback of $ \varrho^{k}_{t} $ 
along $ f  $ once the  functorial  identifications are made.   
Combining  this with the relation $   t  \circ  f  = f_{U} \circ t_{T} $,  we  see   that  $ (\oplus_{i=1}^{k} f_{i}^{*})\circ \varrho^{k}_{t} \circ  t^{*} =  \varrho^{k}_{t_{T}}   \circ  (t_{T})^{*} \circ f_{U}^{*}  $.  So 
\begin{align*} f_{k}^{*}( {}_{ \alpha}  \mathrm{Eis}_{\QQ_{p}}^{k}(t) )   &  =  f_{k}^{*} \circ \pr^{k} \circ \varrho_{t} ^ { k  }     \circ t^{*}   ( {}_{\alpha} \mathrm{pol} _{ \QQ_{p}}   ^ { k }  ) \\
& = \mathrm{pr}^{k} \circ \varrho_{t_{T}}    ^ { k   }   \circ t_{T}^{*} \circ f_{A}^{*} (  {}_{ \alpha}  \mathrm{pol} _ { \QQ_{p}} ^ { k }  )         =  {}_ {\beta} \mathrm{Eis} ^{k} _ {\QQ_{p}} (t_{T})   
\end{align*}
which proves the first claim.
For the second claim,  note that since $ f_{D}^{*} \circ \pi_{D}^{*}  = \pi_{D_{T}}^{*} \circ f^{*} :  \mathrm{H}^{0}_{\et}(  S, \QQ_{p} ) \to  \mathrm{H} 
   ^{0}_{\et}(D, \QQ_{p}) $ and $ f^{*}(1) = 1 \in \mathrm{H}_{\et}^{0} ( T , \QQ_{p} ) $, we have $ f_{D}^{*} ( \pi_{D}^{*}(1)) = \pi_{D_{T}}^{*}(1)   $.   So  it  suffices  to  show  that  $ f_{D}^{*} \circ e_{D,*}(1) = e_{D_{T}, *}(1) $.  This is easily seen by moving to an \'{e}tale cover of  $ S $ on which  $ D $ is trivialized. 
\end{proof}The next  result is an analogue of \cite[Lemma 1.7(2)]{kkato}.
Let $  \varphi  : A \to A '  $ be   a     $S$-isogeny  and   
$ D' $, $ U' $, $ \tilde{U} $, etc.,    be as in \S \ref{normcompsec}.
For the result below, we assume  that $  s : =  \varphi \circ t \neq e' $ and that  $  \ker  \varphi  $ is a  constant    $S$-group scheme over a  finite   abelian   group $ \Gamma $, so that the structure map $ f : \ker \varphi \to S $ is identified with  
$ \sqcup_{\gamma \in \Gamma } \mathrm{id}_{S} $.         
For any $ \gamma \in  \Gamma $,  we  denote $ e_{\gamma} : S \to \ker \varphi $ the section indexed by $ \gamma $ and set $ t_{\gamma} := t + i \circ e_{\gamma} $ where $ i   :      \ker \varphi \to A $ denotes the inclusion.       Finally,  let     $$  \mathrm{Sym}^{k}\varphi_{*}  :  \mathrm{H}^{2d-1}_{\et}\big(  S,  \mathrm{Sym}^{k}   ( \mathscr{H}_{A,\QQ_{p}} )  (d)   \big)  \to   \mathrm{H}^{2d-1}_{\et} \big ( S ,  \mathrm{Sym}^{k}(\mathscr{H}_{A', \QQ_{p}} ) (d)    \big )  $$
denote the map induced by $ \varphi  $.   
\begin{lemma}   \label{Eisnormcomp}          
$  { } _ {   \beta }  \mathrm{Eis} ^{ k } _ { \QQ_{p}} ( s )  
=    \sum _ { \gamma \in \Gamma }          \mathrm{Sym}^{k} \varphi_{*}    \big  (   { } _ {   \alpha   }     \mathrm{Eis}^{k}_{\QQ_{p}}(t_{\gamma})  \big )  $ where $ \beta  =  \varphi_{D, *} (\alpha) \in \QQ_{p}[D ' ]^0 $.    
\end{lemma}

\begin{proof}  As $ s  \neq e'  $  is $ N $-torsion  and $ (N,c) = 1 $,  $  s  $ factors via $ U' $ and therefore each $ t_{\gamma }$ factors via $ \tilde{U} $.  Let $ \tau : \ker \varphi \to  \tilde{U} 
$ be the morphism which  equals  $  t_{\gamma} $ on the component indexed by $  \gamma  $.      We claim     that  the  diagram 
\begin{equation}   \label{Eiscartesian}      
\begin{tikzcd}[]   
\tilde{U}  \arrow[r, "\tilde{\varphi}"]  &   U'   \\
 \ker  \varphi  \arrow[u ,  " \tau    " ]    \arrow[r, "f"]  
 &  S  \arrow[u, "s"']       
\end{tikzcd}   
\end{equation}
is Cartesian  in $ \mathbf{Sch} $. For this, we may replace $ \tilde{\varphi} $ with $ \varphi : A \to A' $. So  let   $ p : X \to A  $, $ q : X \to   S   $ in $ \mathbf{Sch} $ be  such that $  \varphi   \circ p = s  \circ  q  $. Assume  wlog that $ X $ is connected. Since  $ \pi \circ p = \pi '  \circ  \varphi  \circ p  = \pi '  \circ s  \circ q = q $, $ p $ is  an   $ S $-morphism. 
Since  $   \varphi   \circ  ( p - t  \circ  q ) = s  \circ  q -  s  \circ  q  = 0 $, there is a unique  $S$-morphism $ \kappa : X \to \ker  \varphi  $ such that $ p - t \circ q =    i     \circ  \kappa  
$. Since $ X $ is connected, there is a unique $ \delta \in \Gamma $ such that $ \kappa $ factors as $ X \xrightarrow{ q } S \xrightarrow{ e_{\delta} }  \ker \varphi $. So $  p   =  i   \circ     \kappa   + t \circ  q =   t_{\delta} \circ  q  = \tau    \circ \kappa $ and the universal property is   verified.   

Since (\ref{Eiscartesian}) is Cartesian, we have   isomorphisms  $  \sqcup_{ \gamma } s^{*} =    f^{*}  s^{*}  \simeq  \tau^{*} \tilde{\varphi}^{!}$ and $  s  ^{*}  R \tilde{\varphi}_{*} \xrightarrow{\sim
} Rf_{*} \tau^{*} $.  In particular,  $ R s _{*} s ^{*} R \tilde{ \varphi  } _{*}     \tilde{ \varphi } ^{!}   \xrightarrow{\sim} R (sf)_{*} (sf)^{*} =   \sqcup_{\gamma }  Rs_{*}s^{*}   $. Using this, we obtain a commutative diagram  

\begin{center}  
\begin{tikzcd}[sep = large]  
 \mathrm{H}^{2d-1}_{\et}\big (U, 
 \mathcal{L}og^{(k)}_{U,\mathbb{Q}_{p}}(d) \big ) \arrow [d, swap ,"{\oplus t_{\gamma}^{*}}"]   \arrow[r,  " 
  \jmath^{*} \circ  \varphi_{\#}    "] & \mathrm{H}^{2d-1}_{\et}\big (\tilde{U} , 
\tilde{\varphi} ^{*}  \mathcal{L}og^{(k)}_{ U  ',\mathbb{Q}_{p}}(d) \big )   \arrow[r,   "\tilde{ \varphi } _{*}  "   ]   \arrow[d, "\oplus s^* 
"]           &   \mathrm{H}^{2d-1}_{\et}\big (U', 
 \mathcal{L}og^{(k)}_{U',\mathbb{Q}_{p}}(d) \big )    \arrow [d, "{ s^ {*}}"]   \\    
{\bigoplus_{\gamma}\mathrm{H}^{2d-1}_{\et}\big (S,t_{\gamma}^{*}\mathcal{L}og^{(k)}_{U,\mathbb{Q}_{p}}(d)   \big    )  } \arrow[r, "{ (t_{\gamma}^{*} \varphi_{\#})_{\gamma} }"] &  
 \bigoplus_{\gamma}\mathrm{H}^{2d-1}_{\et}\big (S,s^{*}\mathcal{L}og^{(k)}_{U',\mathbb{Q}_{p}}(d)   \big    )   \arrow[r, "\sum"]     
   &   \mathrm{H}^{2d-1}_{\et}\big (S, s ^{*}\mathcal{L}og^{(k)}_{U',\mathbb{Q}_{p}}(d)   \big    )      
\end{tikzcd}
\end{center}
where maps in   the     right square are induced by various adjunction    transformations    applied to $  \mathcal{L}og^{(k)}_{U',\QQ_{p}}(d) $,      e.g., the middle vertical arrow is induced by $  R \tilde{ \varphi } _{*}   \tilde { \varphi }  ^{!}  \to     R s_{*}  s^{*}    R   \tilde {  \varphi  } _{*} \tilde { \varphi  }    ^{!} $.  The composition  given by the    top row   is just $  \mathscr{N}_{\varphi} $ by the discussion in \S \ref{normcompsec}. So by  Proposition   \ref{poltrace}, it suffices to show that for each $ \gamma  \in  \Gamma  $, the composition $ \varrho_{s}^{k}     \circ t_{\gamma}^{*}  \varphi_{\#}  \circ  (\varrho^{k}_{t_{\gamma}})   ^ { - 1 }       $ is equal to $ \mathrm{Sym}^{k}\varphi_{*} $ on the $k$-th symmetric power. 
But   this     composition is easily seen  to be $ e^{*} \varphi_{\#} $ once the maps at the ends  are written in terms of an isogeny $  \psi :A' \to A''  $ that annihilates $ s $ (see the proof Corollary \ref{splittingprinciple}). That claim  then  follows by 
Proposition \ref{functoriality}. 
\end{proof} 

\subsection{Interpolation}            
\label{momentmapsec}

In this subsection,  we recall the definition of integral logarithm sheaf and state Kings' result on the interpolation  of Eisenstein classes.  For   proofs,  we  refer the reader to \cite[\S4-6]{Kings15}.  

Let $ \pi : A  \to S $ be as in  \S  \ref{Qppolylogsection}. Recall 
(\ref{HrisApr}) that for each positive integer $ r $, $ \mathscr{H}_{r} := \mathscr{H}_{\Lambda_{r}} $  is isomorphic to the representable sheaf  associated with $ p^{r} $-torsion subscheme $   X_{r} : = A[p^{r}] $. 
Let  $ \pr_{r} : X _{r} \to S $ denote the structure map and set  
\begin{equation}  \label{lambdarHr} \Lambda_{r} [\mathscr{H}_{r}  ]  = \Lambda_{r}[X_{r}]     :=  \pr_{r,*} 
{  \Lambda _{r}   }     \in    \mathbf{Sh} ( S_{\et} )   
\end{equation} 
Then $ \Lambda_{r}[\mathscr{H}_{r}] $ is a sheaf of (abelian) group algebras over $ \Lambda_{r} $ with product $ \mathscr{H}_{r} \times \mathscr{H}_{r} \to \mathscr{H}_{r} $ induced by the group structure on $ X_{r} $.  More generally,  let $ t : S \to A $ be  a  torsion  section and let $  \pr_{r} :  X_{r} \langle t   \rangle  \to  S $ be defined   by the pullback diagram \vspace{-0.1em} \begin{center}
\begin{tikzcd} 
X_{r}   \langle t  \rangle  \arrow[r]  \arrow[d, "{\pr_{r}}"'] &  A  _ { r }    \arrow[d, "{[p^{r}]}"]    \\
    S  \arrow[r, "t"]   &   A  
\end{tikzcd}  
\end{center}
where $ A_{r} :  = A $ considered as a finite \'{e}tale cover of $ A $.
We denote by $ \mathscr{H}_{r} \langle t  \rangle   $  the    representable  sheaf   corresponding to $ X_{r}  \langle t \rangle $ and  define  
\begin{equation}   \label{lambdaHrt}          \Lambda_{r}[ \mathscr{H}_{r} \langle  t \rangle ] = \Lambda_{r} [ X_{r}  \langle t  \rangle  ]  : =  \pr_{r,*} \Lambda_{r}  \in  \mathbf{Sh}(S_{\et} ) 
\end{equation} 
Then $ \Lambda_{r} [ \mathscr{H}_{r} \langle e \rangle ] =  \Lambda _ { r }  [  \mathscr{H}_{r} ] $. Since the cover $  A_{r} $ is a $ X_{r} $-torsor  on $ A $, $ X_{r} \langle t   \rangle   $ is a $ X_{r} $-torsor over $ S $. This makes   $  \Lambda_{r}  [ \mathscr{H}_{r}  \langle t  \rangle  ] $  a sheaf of modules over $ \Lambda_{r} [ \mathscr{H} _{r}  ]   $.  Let  $ \lambda_{r} =  \lambda_{r} \langle t \rangle : X_{r+1} \langle t   \rangle  \to  X_{r}  \langle   t    \rangle  $ be the map induced by the universal property of $ X_{r} \langle t \rangle $   applied  to      $ X_{r+1} \langle t \rangle  \to A_{r+1} \xrightarrow{[p]} A_{r} $ and $ \pr_{r+1} $.   Then $ (X_{r} \langle t  \rangle )_{r} $ forms    a 
    pro-system of finite  \'{e}tale covers of $ S $.   
The adjunction $ \lambda_{r,*} =  \lambda_{r,!} \dashv 
\lambda_{r}^{!} = \lambda_{r} ^ { * }     $ (by \'{e}taleness of $ \lambda_{r}$) gives us a map      $    \lambda_{r, * }  
\Lambda_{r+1} = \lambda_{r,!} \lambda_{r}^{!}   \,  
\Lambda   _{r+1} \to {   \Lambda _{r+1 }   }     $    
and post composing with $ \pr_{r,*} $ gives us a map  $ \Lambda_{r+1}[ \mathscr{H}_{r+1}   \langle  t   \rangle    ]  =   \pr_{r , * }  \lambda_{r , *}   \,   \Lambda  _{r+1} \to  \pr_{r,*}   \,  \Lambda _{r+1} $. Reducing modulo $ p^{r} $,    we   obtain  an  induced ``trace"  map   \begin{equation}    \label{tracemaplambdaHr}       \mathrm{Tr}_{r    } : \Lambda_{r+1} [ \mathscr{H} _ {r+1}   \langle  t  \rangle  ] \to  \Lambda_{r} [ \mathscr{H} _{r}  \langle t  \rangle    ]   
\end{equation} 
which  is    compatible with the   underlying        module  structures.   
\begin{definition} The  \emph{sheaf of Iwasawa algebras of $ \mathscr{H}$ on $S$} is defined to be the pro-system $ \Lambda(\mathscr{H}) :=  ( \Lambda_{r}[\mathscr{H}_{r}] )  _  {r  \geq   1   } $ with transition maps  given by   (\ref{tracemaplambdaHr}) for $ t = e  $.      The \emph{sheaf of Iwasawa modules associated with $ t $}  is defined to be the pro-system  $  (  \Lambda_{r}[  \mathscr{H}_{r}  \langle  t  \rangle  ]  )  _ { r \geq  1   }     $.\end{definition}   
For each  non-negative integer $ k $, let $  \Gamma_{k}(\mathscr{H}_{r}) $ denote  the sheafification of the presheaf that sends an open subscheme $ U \subset S $ to the $ k $-th divided power algebra $ \Gamma_{k}( \mathscr{H}_{r} (U) ) $. Then the reduction maps $ \mathscr{H}_{r+1} \to \mathscr{H}_{r} $ induce isomorphisms    $ \Gamma_{k} ( \mathscr{H}_{r+1} ) \otimes _ { \ZZ / p^{r+1} \ZZ } \ZZ / p^{r} \ZZ  \simeq \Gamma_{k} ( \mathscr{H}_{r} ) $ and we obtain a $ \ZZ_{p} $-sheaf  
$ \Gamma_{k}(\mathscr{H}) :=   \big  (\Gamma_{k}(\mathscr{H}_{r}   )       \big    )_{r \geq  1} $.  There is a canonical map   
\begin{equation} \gamma_{k} :  \mathrm{Sym} ^{k}  ( \mathscr{H} )  \to   \Gamma_{k}  ( \mathscr{H} )   \label{canonicalisogaammaH}  
\end{equation}  
induced by   sending     $ m ^{\otimes k} \in \mathrm{Sym}^{k} ( \mathscr{H} )$ to $ k! m^{[k]} $ for $ m $ a  section of $  \mathscr{H} $.    It induces an  isomorphism  $   
\mathrm{Sym}^{k} ( \mathscr{H} ) \otimes \QQ_{p}  \to  \Gamma_{k} ( \mathscr{H} ) \otimes \QQ_{p}  \simeq \Gamma_{k} ( \mathscr{H}_{\QQ_{p}    }    ) $  between   the   corresponding   $ \QQ_  { p } $-sheaves.  

The sheaf $ \mathscr{H}_{r} \in  \mathbf{Sh}   (  S _{\et}  )  $  possesses over $  X_{r} 
\in S_{\et}  $  a  tautological  section $ \tau_{r} \in \Gamma( X_{r} , \mathscr{H}_{r} ) = \mathscr{H}_{r}( X_{r} ) = \mathrm{Hom}_{S}(  X_{r}  , X_{r} ) $ corresponding to the identity map $ X_{r} \to X_{r} $.  Its  $ k $-th divided power gives rise to a section $ \tau^{[k]}_{r}  \in \Gamma( X_{r} ,  \Gamma_{k}(\mathscr{H}_{r}) )   $.  Let $ \Gamma_{k}(\mathscr{H}_{r})  _ { \mid X_{r}  } :  =  \pr_{r}^{*}     \,   \Gamma_{k }  (   \mathscr{H}_{r} ) $  denote the restriction of $ \Gamma_{k}(\mathscr{H}_{r} ) $ to $ X _{r} $.    Then 
\begin{align*} \Gamma (X_{r} ,      \Gamma_{k} ( \mathscr{H}_{r} ) ) =  \mathrm{Hom}_{X_{r}} ( \Lambda_{r},  \Gamma_{k} ( \mathscr{H}_{r} ) _ { \mid  X   _{r} }   )    &      \simeq   \mathrm{Hom}_{S} ( \pr_{r, !  }    \Lambda_{r}   ,  \Gamma_{k} ( \mathscr{H}_{r} )  )   \\
&  \simeq  \mathrm{Hom}_{S} (  \Lambda_{r}  [  \mathscr{H}_{r} ] , \Gamma_{k}   (  \mathscr{H}_{r}  )    ) 
\end{align*} 
where the penultimate isomorphism follows via the adjunction $  \pr_{r, !}  \dashv      \pr_{r } ^{!} = \pr_{r} ^{*}  $  and the last by the identification     $ \pr_{r, !} = \pr_{r ,  *} $.  
Thus  $  \tau_ { r } ^ { [ k ] } $ corresponds to a morphism     
\begin{align} \label{momentkr}   \mathrm{mom}^{k}_{r} : \Lambda_{r} [\mathscr{H}_{r}]     \to    \Gamma_{k}  ( \mathscr{H}_{r} ).
\end{align}  
For fixed $ k $ and  varying $ r $, the maps  $ \mathrm{mom}_{r}^{k} $ are   compatible with respect to  $ \mathrm{Tr}_{r} : \Lambda_{r+1}  [ \mathscr{H}_{r+1} ] \to   \Lambda_{r}[\mathscr{H}_{r} ] $ and the reduction maps     $ \Gamma_{k} (   \mathscr{H}_{r+1} )  \to \Gamma_{k} ( \mathscr{H}_{r+1} ) \otimes _ { \Lambda_{r+1}}  \Lambda_{r}  \simeq \Gamma_{k} ( \mathscr{H}_{r} ) $ (\cite[Lemma 4.5.1]{Kings15}).   
\begin{definition}   \label{momentmapdefi}     The  \emph{$k$-th moment map} is defined to be the morphism   
$ \mathrm{mom}^{k} : \Lambda(\mathscr{H})  \to \Gamma_{k}(\mathscr{H}) $  of  pro-systems   obtained by the compatible system     $  (\mathrm{mom}^{k}_{r})_{r \geq    1    } $ given in (\ref{momentkr}).    
\end{definition}  
Parallel to the construction of $ \Lambda(\mathscr{H}) $  
is   the construction of a pro-sheaf on $ A $ that is the integral analogue of the $ \QQ_{p} $-logarithm pro-sheaf.  
Let  $ [p^{r}] : A_{r} \to A    $ be the $ X_{r} $-torsor over $ A $ as  above.     Denote by $ \mu_{r} : A_{r + 1 }   
\to A_{r} $ the transition map induced   by   
  $ [p] : A \to A  $. Then we have a pro-system $ (A_{r}  ) 
_{r \geq 1} $ of  finite  \'{e}tale  covers 
  on $ A  $.     For $ s \geq 1 $, let $$  \Lambda_{s} [ A_{r} ] : = [p^{r}]_{*} \Lambda_{s}  \in  \mathbf{Sh}(A_{\et}) . $$ 
Then as above, we get morphisms $  \mu_{r,s} : \Lambda_{s}[A_{r+1} ] \to \Lambda_{s} [A_{r}] $ for all $ s $, $ r $.  If $ s = r $,  the     pre-composition of this map with reduction modulo $ p^{r} $ gives a  ``trace"  map  $  \mathrm{Tr}_{r 
} : \Lambda_{r+1} [A_{r+1}]  \to  \Lambda_{r}  [ A_{r} ]  $ as in (\ref{tracemaplambdaHr}).

\begin{definition} 
The \emph{integral logarithm sheaf}  $ \mathcal{L} $ is defined to be the pro-system  $ ( \Lambda_{r }   [ A_{r} ] ) _ { r \geq 1 }   \in \mathbf{Sh}(X_{\et})^{\mathbb{N}} 
$ with transition maps given by $ \mathrm{Tr}_{r}  $. 
\end{definition} 
Via the $ X_{r} $-torsor $ [p^{r}] :  A_{r} \to A $, $ \mathcal{L} $ becomes a free rank $ 1 $ module  over $  \pi^{*}  \Lambda ( \mathscr{H} )  $. The base change compatibility of such pro-systems (\cite[\S 4.4]{Kings15}) implies that $ \mathcal{L} $ is compatible with arbitrary base change. Consequently, there is an isomorphism   
\begin{equation}    \label{fakesplitting}    \varsigma_{t} :  t^{*} \mathcal{L} \simeq   \Lambda (     \mathscr{H} \langle t  \rangle  )
\end{equation}   of  sheaf of modules over  $ \Lambda(\mathscr{H} ) $.  In particular, $ e^{*}  \mathcal{L}  \simeq  \Lambda ( \mathscr{H} ) $. 
   
The integral logarithm sheaf enjoys properties similar to those of $ \Log_{\QQ_{p}} $. Below we record the ones needed to state Kings' result.  
\begin{proposition}  [Splitting principle]     Let $ n $  be a  positive integer  and  $ t : S \to A $ be  an $ n $-torsion section. Then there  exists a   canonical       homomorphism $ [n]_{\#} : t^{*} \mathcal{L}    \to   \Lambda   ( \mathscr{H} )   $ 
which is an isomorphism if $ (n,p) = 1 $.      \label{intsplitting}          
\end{proposition}

\begin{corollary}  \label{integralsplitcoro}  Let $ c $ be an integer prime to $ p $ and $ D = A[c] $.  Then there exists a canonical isomorphism $ \iota_{D}^{*} \mathcal{L} \simeq  \pi_{D}^{*}   \Lambda  (  \mathscr  { H  }  )   $ 
where $ \iota_{D} $, $ \pi_{D} $ are as in   (\ref{geometricsituation}).    
\end{corollary} 

\begin{proposition}    [Vanishing of Cohomology]   Let $ r, s $ be non-negative integers. There exist   natural isomorphisms  $  R^{2d} \pi_{*} (  \Lambda_{s} [ A_{r} ] )  \simeq \Lambda_{s} (-d)  $ which are compatible with respect to $  \mu _{r,s} $ and reduction modulo $ p^{s-1} $. For each  $ i = 0, \ldots,  2 d - 1  $ ,  there exist a sufficiently large integer $ r' $ such that $ R^{i} \pi_{*} \Lambda_{s}[A_{r'}] \to R^{i} \pi_{*} \Lambda_{s} [A_{r} ] $ is zero.  In particular, 
$$    
\mathrm{H}^{i}_{\et} \big  (A,   \mathcal{L}  (d) \big  )  
\simeq    \begin{cases}  0  &  \text{ if } i < 2d   \\  \mathrm{H}^{0}_{\et} (S, \ZZ_{p} )   &  \text{ if }  i = 2d   .
\end{cases} $$ 
\label{vanishingcohomologyintegral}  
\end{proposition}
\begin{remark} By \cite[Proposition 1.6]{Jannsen1988} and  above, we have  $  \mathrm{H}^{i}_{\et}( A , \mathcal{L}(d)) = \varprojlim_ { r }  \mathrm{H}^{i}_{\et}(A ,  \Lambda_{r}[A_{r}]) $ where the limit involves  ordinary \'{e}tale cohomology groups 
(or continuous  groups with constant pro-systems $ \Lambda_{r}[ A_{r} ]$). 
\end{remark} 

Fix as before  an integer $ c  > 1 $ that is   invertible   on $ S $  but which is now also prime to $ p $. We retain   the      notations introduced in diagram   (\ref{geometricsituation}).  Repeating the same 
argument  as  in  \S     \ref{LogQpressection} and invoking Proposition \ref{vanishingcohomologyintegral},  we find an exact sequence        
 \begin{equation}   \label{integrallogresiduesequence}         0  \to   \mathrm{H}    ^{2d-1}_{\et}   \big (  U  ,   \mathcal{L}  ( d )  \big )   \xrightarrow{\mathrm{res}}  \mathrm{H}^{0}_{\et}   \big (D ,  \iota ^ {  *   }    \mathcal{L}   \big  )  \to   \mathrm{H}^{0}_{\et}(S,  \QQ_{p} )    
\end{equation}
where $ \mathrm{res} $ is again  referred to as the \emph{residue}   map.   By Corollary \ref{integralsplitcoro}, we  may replace $ \mathrm{H}^{0}_{\et} ( D   ,  \iota^{*}  \mathcal{L} ) $ with $ \mathrm{H}^{0}_{\et} \big ( D, \pi_{D}^{*}  \Lambda( \mathscr{H} )   \big )      $.   Since we have  a  section $ \mathbf{1} : \ZZ_{p} \to  \Lambda(\mathscr{H}) $ induced by sending 
$ 1 \in  \ZZ/p^{r} \ZZ $ to the identity in $ \Lambda_{r}[A_{r}] $, we have a corresponding  inclusion  $ \mathrm{H} ^{0}(D, \ZZ_{p}) \to   \mathrm{ H }    ^{0}_{\et}( D ,  \pi_{D}^{*} \Lambda  ( \mathscr{H} ) ) $.       Let   $  \ZZ_{p}  [ D ] ^ {0 } $ denote the kernel of the  trace  map $$  \mathrm{H}     ^ { 0 } _ { \et } ( D , \ZZ_{p} )  \to  \mathrm{H}    ^{0} _ { \et } ( S , \ZZ_{p} ) . $$  
Then $ \ZZ_{p}[D]^{0} $ lies in the   image of residue map  for the same reasons as in  \S  \ref{LogQpressection}.      Note  that   the class  $ \alpha _ { c }    $    of     (\ref{alphacchoice})      is a member of $ \ZZ_{p}  [ D ] ^ {0  } $. 
\begin{definition}  Let $ \alpha \in \ZZ _ { p } [ D ] ^{0} $.  The  \emph{integral \'{e}tale  polylogarithm with residue $ \alpha $}  is the unique  class  \vspace{-0.1em} $$ { }   _ {\alpha} \mathrm{pol} _{ \ZZ_{p}}  \in   \mathrm{  H } _{\et}^{2d-1}  \big  (    U, \mathcal{L}(d)   \big )  $$   \vspace{-0.1em}    
such that $ \mathrm{res} \big (  { } _ { \alpha }  \mathrm{pol} _{ \ZZ_{p}   }  )    =  \alpha $. 
\end{definition}   
Fix now an integer  $ N > 1 $ prime to $ c $ and  invertible on $ S $. Let $ t : S  \to U   $ be an $ N $-torsion section.   By  the adjunction $ \mathrm{id} \to R t_{*} t^{*}  $, isomorphism (\ref{fakesplitting}) and  Proposition  \ref{intsplitting}, we  obtain a composition 
\begin{align} 
  \mathrm{H}_{\et} ^{  2 d  - 1 }   \big  ( U    ,   \mathcal{L} ( d )   \big  )  \xrightarrow{t^{*} }   
\mathrm{H}^{2d-1}   _ { \et  }  \big  ( S ,  t^{*}  \mathcal{L}(d)  \big   )  \xrightarrow[\sim]{\varsigma_{t}} 
\mathrm{H}^{2d-1}   _ { \et }   \big  ( S ,          \Lambda  ( \mathscr{ H }    )     \langle  t   \rangle    (d)  \big  )  &   \xrightarrow{[N]_{\#}}   \mathrm{H}_{\et}^{2d-1} \big (S,  \Lambda (  \mathscr{H} )  (d)  \big ) .
    \label{cohomomentmap}       
\end{align}  

\begin{definition} The \emph{Iwasawa-Eisenstein class} $ {}_{\alpha} \mathcal{EI}_{N} ( t)  \in    \mathrm{H}_{\et} ^ { 2 d - 1 }  \big (  S , \Lambda   (  \mathscr{ H }   )  (d) \big ) $
\emph{associated with $ t $, $ N $ and $ \alpha $} is defined to be  image of  $   {      }  _   {  \alpha  }    \mathrm{pol}  _ { \ZZ_{p} } $ under the composition $ [N]_{\#}   \circ  \varsigma_{t}     \circ t^{*} $ of 
   (\ref{cohomomentmap}).          
\end{definition}

The following result of Kings shows that the classes   just  defined 
interpolate  rational Eisenstein classes.  Let $ \mathrm{mom}^{k}_{\QQ_{p}} $ denote the composition 
\begin{align}    \notag    \mathrm{H}_{\et} ^ { 2 d - 1 }   \big (S, \Lambda(\mathscr{H}) (d)    \big  )  \xrightarrow{ \mathrm{mom}^{k} }    \mathrm{H}^{2d-1}_{\et}  \big  ( S , \Gamma_{k}  ( \mathscr{H} )     ( d    )       \big    )   \xrightarrow {   -  \otimes \QQ_{p} }& \mathrm{H}^{2d-1}_{\et}   \big   ( S , \Gamma_{k} ( \mathscr{H} )    (d)   \big   )  \otimes _ { \ZZ_{p}}   \QQ_  { p }  
\\   \xrightarrow[\sim]{  \sigma  _{k }  }    &   \mathrm{H}^{2d-1}_{\et} \big  ( S ,   \mathrm{Sym}^{k}  \mathscr{H}_{\QQ_{p}}  (d )    \big )  \notag  
\end{align}  
where the first map is induced  by the moment map of Definition \ref{momentmapdefi} and   $  \sigma_{k}   $  is  
induced  by  (\ref{canonicalisogaammaH}). 
\begin{theorem}[{\cite[Theorem 6.3.3]{Kings15}}]   \label{Kingsmain}           For $ \alpha \in \ZZ_{p} [D]^{0} $ and $ N > 1 $ an integer relatively  prime to $ c $,   the   $k$-th   rational Eisenstein  class $ {}_{\alpha} \mathrm{Eis}_{\QQ_{p}}^{k}(t) $  along   a  non-zero  $ N $-torsion section   $  t  :   S \to A   $   satisfies $$ \mathrm{mom}^{k}_{\QQ_{p}}  \big (  {}_{\alpha}  \mathcal{EI}_{N}(t)  \big )   =   N  ^  { k   }  {}_{\alpha}  \mathrm{Eis}_{\QQ_{p}} ^{k} (t)  .  $$
In particular, $ N^{k} {}_{c}  \mathrm{Eis}^{k}_{\QQ_{p}}(t) 
$ lies in the image of $ \mathrm{H}^{2d-1}_{\et}   \big   (S, \Gamma_{k}(\mathscr{H}  )   ( d   )      \big )  $ under 
$ \sigma_{k} \circ  ( - \otimes  \QQ_{p} )   $.
\end{theorem}

\section{Siegel modular varieties}  
\label{Siegelsec}    

In this section, we recall the definition and 
moduli interpretation  of Siegel modular varieties  with 
principal   level  structures.  Since Eisenstein classes of Definition  \ref{EisdefiQp} are apriori only defined for   scheme theoretic sections, we need to spell out the effect of various  maps between moduli varieties in a non-adelic  fashion. 
\subsection{The  Shimura  data}      
\label{Siegelreview}  For $ n \geq 1 $ an integer,   let  $ I_{n} $ denote the $ n \times n $ identity matrix. Let  $$ 
J = J_{n} = \begin{pmatrix} & I_{n} \\ -  I_{n}  & \end{pmatrix} \in \mathrm{Mat}_{2n  \times  2n }(\ZZ) $$
be    the standard symplectic matrix and 
$ \mathbf{G} = \mathrm{GSp}_{2n} $ denote the reductive group scheme over $ \ZZ $ whose $ R $ points for a ring $ R $  are given by $ \mathbf{G}(R) = \left  \{  g \in  \GL_{2g}(R) \, | \,  g ^{t}  J g = c(g) J \text { for } c(g) \in R^{\times}  \right \}  $. The induced  homomorphism $ c : \mathrm{GSp}_{2g} \to \GG_{m} $ is called the  \emph{similitude}.  The center of $ \Gb $ is denoted by $ \mathbf{Z} $ which is identified with $ \GG_{m} $ via diagonal matrices.   
Let $ \delT := \mathrm{Res}_{\CC/\RR} \GG_{m}  $ denote the Deligne torus and define     
\begin{align*} 
h _ { \mathrm{std} } :  \delT \to \mathbf{G}_{\RR}  \quad \quad 
(a+b \sqrt{-1}) \mapsto   \begin{pmatrix} a I_{n} & bI_{n} \\   -b I_{n}  &  a I_{n}   \end{pmatrix}.
\end{align*}   
Let  $  \mathcal{X}  $    denote the $ \Gb(\RR) $-conjugacy class of $ h_{\mathrm{std}} $. Then  $ (\mathbf{G}_{\QQ}, \mathcal{X})$  satisfies axioms SV1-SV6 of \cite{MilneShimura} and in particular,  constitutes a Shimura datum. 
Its    reflex field is $ \QQ $. For a neat compact open subgroup $ K \subset \Gb(\Ab_{f} ) $, let $  \mathrm{Sh}(K  )  =  \mathrm{Sh}_{\Gb}(  \mathcal{X}   ,  K)     $ denote the  corresponding canonical model over $ \QQ $. It is a smooth  quasi-projective     variety of dimension $ n(n+1)/2   $   whose  $ \CC $-points are  identified  with   the double quotient $ \Gb(\QQ) \backslash [\mathcal{X}  \times  \Gb(\Ab_{f}) ] / K  $.     

Let $ (V_{\ZZ}, \psi) $ denote the standard symplectic $ \ZZ$-module where $ V_{\ZZ} : = \ZZ^{2n} $ and $ \psi : V _ { \ZZ  }  \times V _ { \ZZ  }   
  \to  \ZZ    $ is the pairing induced by $ J $.  We let $ e_{1}, \ldots, e_{2n} $ denote the standard basis of $ V_{\ZZ} $.  
For any commutative ring $ R $ with unity, we denote $ V_{R}    :=  V_{\ZZ} \otimes_{\ZZ}  R $ and  $ V_{\mathbb{A}_{f}} $ 
simply as $ V_{f} $.       We will view elements of $ V_{R} $ as column vectors and let $ \mathbf{G}(R)  $ act on $ V  _ { R   }      $ by    left matrix multiplication.  For each positive integer $ N $, let  $$ K_{N}  : =   \left \{ g \in \Gb(\Ab_{f} ) \, | \, ( g - 1 ) V_{\widehat{\ZZ}} \subset N V_{\widehat{\ZZ} } \right \} $$ 
be the   \emph{principal  congruence  subgroup} of $ \Gb(\Ab_{f}) $ of level $ N $. These   subgroups      form a  base for the topology of $ \Gb(\Ab_{f} )$ at identity and $ K_{N} \trianglelefteq       K_{1} $ for all $ N \geq  1 $.      The quotient $ \Gb(\Ab_{f})/K_{1} $ is identified with the set of self-dual $ \widehat{\ZZ} $-lattices in $ V_{f}     $ by identifying the coset  $ gK_{1} \in \Gb(\Ab_{f})/K_{1} $ with the lattice  $g V_{\widehat{\ZZ}} $.     More   generally,    $ \Gb(\Ab_{f})/K_{N} $ is identified  with the set of pairs $ (\widehat{H}, \bar{\eta}) $ where $ \widehat{H}  \subset V_{f}     $ is a self-dual  $ \widehat{\ZZ}$-lattice and 
$ \bar{\eta}  : V_{\ZZ}/ N V_{\ZZ}  \xrightarrow{\sim}  \widehat{H} / N \widehat{H} $ is a choice of a  symplectic isomorphism. 

\subsection{Moduli interpretation}     \label{moduli}     For $ N \geq 3 $, $ \mathrm{Sh}(K_{N}) $ is the $ \QQ $-scheme representing the following moduli problem. Let $ \mathbf{Sch} _ { \QQ   } $ denote the category of  locally  Noetherian $ \QQ $-schemes and let $ \mathfrak{M}_{N} : \mathbf{Sch}  _ { \QQ  }    \to \mathbf{Sets} $ 
denote the   contravariant    functor that sends a $ \QQ $-scheme $ S $  to   isomorphism  classes  of    triples $ (A, \lambda,  \eta ) $ where 
\begin{itemize}  
\item $ A $ is an abelian scheme over $ S $ of  relative dimension $ n $, 
\item $ \lambda : A \xrightarrow{\sim} A^{\vee} $ is a principal polarization,    
\item $ \eta :  ( V_{\ZZ} / N  V _{\ZZ} ) _ { S  }  \xrightarrow{\sim}  A[N]   $ is a symplectic similitude of group schemes over $ S $, i.e.,      $ \eta $ is an isomorphism     of $S$-group schemes such that  the Weil pairing on $ A[ N   ]     $ corresponds to a  $ ( \ZZ / N \ZZ ) ^{ \times } $-multiple of the pairing $ \psi $ after fixing an identification of $ ( \ZZ / N \ZZ )^{\times} $ with the roots of unity $ \mu_{N} $. 
\end{itemize}
The isomorphism $ \eta $ is also referred to as a \emph{principal level $  N   $ structure}.  Given a morphism $ f : T \to S $ of schemes, the morphism $  \mathfrak{M}_{N}(S) \to  \mathfrak{M}_{N}(T)  $   is given by pullback  of families.  To say that $ \mathfrak{M}_{N} $ is represented by $ \mathrm{Sh}(K_{N}) $ is to say  that  there exists a natural   isomorphism  $$  \Psi_{N} :  \mathfrak{M}_{N} \to \mathrm{Hom}_{\mathbf{Sch}_{\QQ}}( - , \mathrm{Sh}(K_{N} ) ) $$ of  functors on $ \mathbf{Sch}_{\QQ}   $.    Abusing notation,  we denote by $ \mathrm{id}_{N} \in  \mathfrak{M}_{N}(\mathrm{Sh}(K_{N})) $ the isomorphism class that corresponds to the identity map $  \mathrm{id}_{N} :   \mathrm{Sh}(K_{N}) \to \mathrm{Sh}(K_{N}) $.  
By definition,  there is   an abelian scheme $  \pi_{N  }    : 
\mathcal{A}_{N} \to \mathrm{Sh}(K_{N}) $ with   principal          polarization $ \lambda_{N}^{\mathrm{univ}}  $ and  principal level $ N $-structure $ \eta_{N}^{\mathrm{univ} } $ such that the isomorphism  class of $ ( \mathcal{A}_{N} ,  \lambda_{N}^{\mathrm{univ}} ,  \eta_{N}^{\mathrm{univ}} ) $ equals  $ \mathrm{id}_{N}  $.      It is   referred to as the   \emph{universal family} on $ \mathrm{Sh}(K_{N}) $. We fix such a family once and for all for each $ N \geq 3 $. For $ v \in V_{\widehat{\ZZ}}$, we let $$ t_{v, N} : \mathrm{Sh}(K_{N}   ) \to  \mathcal{A}_{N} $$ 
denote the canonical  $ N$-torsion section that corresponds under $ \eta_{N}^{\mathrm{univ}} $ to the class of $ v $ in $   V_{\widehat{\ZZ}} /  N V_{\widehat{\ZZ}} =  V_{\ZZ} / N  V_{\ZZ}  $. 

Let $ M , N \geq 3 $ be integers such that $ M  | N   $ and let $ g \in \Gb(\Ab_{f}) $ be an element such that $ K_{N} \subset  g K_{M} g^{-1} $ that we fix for this paragraph.   There  is   a      finite  \'{e}tale  map $ [g]_{N,M} : \mathrm{Sh}(K_{N}) \to \mathrm{Sh}(K_{M}) $ of $ \QQ$-schemes     given on complex  points by right multiplication by $ g  $   in  the   second  component.  It induces a  natural  transformation  $ \mathrm{Hom}_{\mathbf{Sch}_{\QQ}} ( - , \mathrm{Sh}(K_{N}))  \to  \mathrm{Hom}  _ { \mathbf{Sch} _ { \QQ }   }    ( - , \mathrm{Sh}(K_{M})) $, corresponding to which there is a  natural transformation $$ {} ^ {g}  \Phi_{N, M}  : \mathfrak{M}_{N} \to  \mathfrak{M}_{M} .  $$
We  explicitly  describe the effect of   ${}^{g}\Phi_{N,M} $ 
for certain $ g $, $ N $, $ M $ (cf.\  \cite[p.9]{Laumon}).  
Assume that $ g $ is such that 
 $   g  (d  V _{\widehat{\ZZ}})     \subset   V_{\widehat{\ZZ}}  \subset   g   V_{\widehat{\ZZ}}        $   where $  d :  = N/M  $. 
 Let     $ \iota : V_{\ZZ} / M  V _{\ZZ}  \hookrightarrow  V _{\widehat{\ZZ}} / g (  N   V    _{\widehat{ \ZZ}}) $ be the  inclusion given by  
$$   \iota :   V _{\ZZ} / M V _{\ZZ}  = d V _{\widehat{\ZZ}} / N V _{   \widehat{ \ZZ }   }  \xrightarrow[\sim]{   g  \,    \cdot  }    g (d V _{\widehat{\ZZ}})    / g  (  N V_{\widehat{\ZZ}} )           \hookrightarrow    V _ { \widehat{ \ZZ } } /  g  (  N   V   _{ \widehat{\ZZ}} ) $$ 
and let   $ \gamma    :  V _{\ZZ}/N  V  _{\ZZ} \to V _{\ZZ}  / N   V  _{\ZZ}  $  be   the   symplectic     endomorphism  induced by  $ v \mapsto g^{-1}  v $ for $ v \in V_{\widehat{\ZZ}}   $ (which 
is well-defined    since $ g^{-1} V_{\widehat{\ZZ}} $ is contained in $ V_{\widehat{\ZZ}} $).       Note  that    $$ \ker \gamma =   g ( N   V _{\widehat{\ZZ}} )  / N V _{\widehat{\ZZ}} \quad   \text{ and } \quad  ( V_{\ZZ}     /N  V_{\ZZ})/  \ker \gamma    \simeq     V _{\widehat{\ZZ}}/  g ( N  V  _ {\widehat{\ZZ}} ) .   $$ 
Thus $ V_{\ZZ} / M V_{\ZZ} $ embeds into the coimage of $ \gamma $ via  $ \iota $.  
We note  that   $ \ker   \gamma     $ is    a   totally     isotropic   subspace of $ V _{\ZZ} / N  V _{\ZZ}  $ with respect to the induced symplectic pairing and its cardinality equals $ k^{n}$ for some positive integer $ k = k_{\gamma} $. Now let  $ S \in \mathbf{Sch}_{\QQ} $ and     $  (  A_{N} ,   \lambda_{N}  ,  \eta_{N} ) \in  \mathfrak{M}_{N}(S)   $ be (a triple representing) an isomorphism class. Let    $ C_{\gamma}  $ be the finite flat subgroup scheme of $ A_{N} $ (over $S$) given by the image of $ ( \ker \gamma )_{S} $ under $ \eta_{N} $ and let $ \bar{\eta}_{N} :  V_{\widehat{\ZZ}} / g (NV_{\widehat{\ZZ}})_{S} \to A_{N}/C_{\gamma}  $ be  the  embedding  on  quotients   induced by $ \eta_{N} $. By $ \iota_{S} $, we denote  the morphism of constant group schemes over  $ S $ determined  by    $ \iota $.   
Consider the triple     $ ({}^{g}A_{M} , {}^{g}\lambda_{M}, {}^{g}\eta_{M} )  $  where   
\begin{itemize} 
\item  $ {}^{g} A_{M} $ is obtained by the quotient map   $   \psi_{\gamma} :  A_{N} \to A_{N}    /   C  _  {  \gamma  } $, 
\item $ {}^{g}  \lambda_{M } : {}^{g}A_{M} 
\to {}^{g}A^{\vee} $ is the unique   principal   polarization  satisfying $ [k] \circ \lambda_{N}  =  \psi_{\gamma}^{\vee}  \circ  {} ^{g} \lambda_{M}  
\circ   \psi_{\gamma} $  (see \cite[Proposition 13.8]{abvar}), 
\item  $   ^{g} \eta_{M}   : (V_{\ZZ}/ M V_{\ZZ})_{S} \xrightarrow{\sim}   { } ^ { g }     A_{M}  [M]  $ is given by the composition $ 
   \bar{\eta}_{N}  \circ \iota_{S}  $. 
\end{itemize}  
Then   $ (    {  } ^ { g}    A_{M } ,  { } ^ { g }   \lambda_{M} ,  {  } ^ { g }    \eta_{M} ) $  represents  the 
   isomorphism   class    
$ {}^{g} \Phi_{N,M}( A_{N} , \lambda_{N} , \eta_{N} )  \in   \mathfrak{M}_{M}(   S    )   $. 

It will be useful to note a few special cases  of the aforementioned description. First, if $ g $ is identity, the corresponding map on triples is given by ``forgetting"   the level $ N $ structure, i.e., by restricting $ \eta_{N} $ to $ dV_{\ZZ}/ N V_{\ZZ} $.  
Second, if $ g = \kappa \in K_{1} $ and $ N = M $,  the class of a triple $ (A_{N}, \lambda_{N}, \eta_{N} $) is sent to the class of $ ( A_{N} , \lambda_{N}, \eta_{N}   \circ \kappa   
)   $. This induces a right action of $K _{1}  = \mathrm{GSp}_{2n}(\widehat{\ZZ})$ on $ \mathfrak{M}_{N}(S) $. Third,     if $ g $ is an element of the center $ \mathbf{Z}(\QQ) $ (satisfying the conditions of the discussion),  the isogeny $  \psi   _  {\gamma}  :  A_{N} \to {}^{g}    A_{M} $ factors via an isomorphism $ A_{N} \simeq { } ^ { g} A_{M }     $ since  $     C_{\gamma}  $ is the kernel of $ [k] : A_{N} \to A_{N} $ and  $ A_{N}/\ker [k] \simeq A_{N} $.  Using this,  we see  that the map    $ {}^{g}\Phi_{N,M} $ for $ g \in \mathbf{Z}(\QQ) $  is again the forgetful one. This can also be seen directly by the complex uniformization  of these varieties.

We  apply the above description to  universal families and record some observations.       For an arbitrary  $ g \in \Gb(\Ab_{f})   $ satisfying $ K_{N} \subset g K_{M} g^{-1}   $,  
there is    a triple $$({}^{g} \mathcal{A}_{M} , {}^{g} \lambda_{M} ,   { } ^ { g} \eta_{M}) \in  \mathfrak{M}_{M}( \mathrm{Sh}(K_{N}) ) $$ 
corresponding  to $ [g]_{N,M}  \in   \mathrm{Hom}_{\mathbf{Sch}_{\QQ}}(\mathrm{Sh}(K_{N}),  \mathrm{Sh}(K_{M})  ) $ under $ \Psi_{M}    $. By definition, this triple   
is obtained by pulling back the universal family on $ \mathrm{Sh}(K_{M}) $ along $ [g]_{N,M}   $.   But $
 [g]_{N,M}  =   \Psi_{M}    \circ   {}^{g}  \Phi _{N , M }  ( \mathrm{id}_{N} ) $.  The preceding  discussion implies  that  
when $  g  \in  \Gb(\Ab_{f})     $ is such that    
$ ( d    V  _{\widehat{\ZZ}}  )   \subset   V_{\widehat{\ZZ} }  \subset   g  V _{\widehat{\ZZ}}   $,   
there is an isogeny \begin{equation}  \label{univisog}  {}^{g}  \psi_{N, M} ^ { }   :  \mathcal{A} _ { N }   \to  {}^{g} \mathcal{A}  _ { M }   
\end{equation} 
over $ \mathrm{Sh}(K_{N}) $   such  that the   tautological  section $ [g]_{N,M} ^{* }  ( t_{v,M} ) : \Sh(K_{N}) \to {}^{g} \mathcal{A}_{M} $  induced by  
$ t_{v,M} $ 
is   $ {}^{g} \psi_{N,M} \circ t_{gdv, N}$. When $ g = \kappa \in K_{1} $ in  particular (so that $K _{N} \subset K_{M} $),    the  discussion above gives us a  pullback diagram 
\begin{equation}   \label{pullbackdiagram}    
    \begin{tikzcd}[sep = large]    
    \mathcal{A}_{N}   \arrow[r]   \arrow[d,  "\pi_{N}"'   
    ]       &   \mathcal{A}_  { M  }    \arrow[d, "{\pi_{M}}"]        \\ 
    \mathrm{Sh}(K_{N})   \arrow[r , "    \kappa _ { N , M }  " ]  &   \mathrm{Sh}(K_{M})    
    \end{tikzcd}   
\end{equation}  
and   the   tautological section of $ \pi_{N} $ induced by $  t_{v, M} $ for $  v \in  V  _ { \widehat{\ZZ}   } $ 
equals $  t_{\kappa dv , N}   $. 
 
\begin{remark}   \label{twistsimplification}     
We note that a general $ g  \in \Gb(\Ab_{f}) $ can be written as $ z_{a}   h $  where  $  h \in \Gb(\Ab_{f}) $ satisfies $    V _{\widehat{\ZZ}} \subset h     V  _{\widehat{\ZZ} }   $ and  
$ z_{a} \in \mathbf{Z}(\QQ) \simeq  \QQ^{\times} $ is identified with a  positive integer $ a $.   Given $ M \geq 3 $, we can find $ N $ a sufficiently large  multiple of $ M $ such that 
$  h  (N/M) V_{\widehat{\ZZ}}      \subset V _{\widehat{\ZZ}} $.   %
Since $ z_{a} : \mathrm{Sh}(K_{N}) \to \mathrm{Sh}(K_{N}) $ is the identify morphism, the effect of $ {}^{g} \Phi_{N,M} $ can be  described by $ {}^{h}   \Phi_{N, M } $. 
\end{remark}

\section{Parametrization} 
\label{parametrization}    

In this section, we prove our main result in    Theorem \ref{mainpararesult}.    We first need some  preliminaries    however.  

\subsection{Symplectic Orbits} For this   subsection,   we use the 
 notations  introduced  in      \S \ref{Siegelreview}. In particular, $ \Gb = \mathrm{GSp}_{2n} $ denotes the symplectic $\ZZ$-group  scheme of rank $ n $.          
\begin{lemma}   \label{GZorbits}  Let $ R $ be a Euclidean domain.   The map that sends an ideal of $ R $  to the  $ \Gb(R) $-orbit of $ \alpha e_{1} \in  V_{R} $ for $ \alpha  $ a  generator of the ideal establishes  a   bijection between the set of  ideals of $ R $ and    the   orbit  space      $ \Gb(R) \backslash V_{R} $. 
\end{lemma}

\begin{proof} Pick a $ v \in V_{R} $.  Write $ v =  a_{1} e_{1} + \ldots + a_{2n} e_{2n}       $ and let  $ \alpha \in R  $ be a generator of the ideal generated by $  a_{1}, \ldots, a_{n} $.   For any $ x \in R $ and $ 1 \leq i \leq n $, let $ A_{i,n+i}(x) \in \mathrm{Mat}_{2n \times 2n}(R) $ be the matrix that has $ 1 $'s on the diagonal, $ x $ in the $ (i,n+i) $ entry and $ 0 $'s elsewhere. Then $ A_{i,n+i}(x) \in \Gb(R) $ and the action of $ A_{i, n+i}(x) $ on $ v  $     replaces the $ i $-th coordinate    with $ a_{i} + x a_{n+i}  $ while keeping everything else the same.   The matrices $ B_{i,n+i} $ obtained by switching $ i $-th and $ (n+i) $-th row of the $ 2n \times 2n $ identity matrix also  lie in $ \Gb(R) $ and their action on $ v $ is given by switching  the $ i $-th and $ (n+i) $-th coordinates. By using these matrices,       we     can apply the Euclidean algorithm to  replace the $i$-th coordinate of $ v $ with a generator for $ (a_{i}, a_{n+i}) $ for all  $ 1 \leq i \leq n $ and make the remaining entries  of $ v $ equal to    zero. Since $ \Gb(R) $ also contains matrices of the form $$ \begin{pmatrix} A &  \\ & (A^{t})^{-1} \end{pmatrix} $$ 
for any $ A \in \GL_{n}(R) $, we easily deduce that $ \alpha e_{1} \in   
\mathbf{G}(R) v $. Clearly $ \alpha_{1} e_{1} , \alpha_{2} e_{1} \in V_{R}  $ are in the same $ \mathbf{G}(R) $-orbit only if $ \alpha_{1}$ and $ \alpha_{2}$ generate the same ideal.    
\end{proof}  

For  $ R $ as above and $ v \in V_{R} $, let $I_{v} \subset R $ denote the ideal generated by its standard coordinates.  For  $ I $ an ideal of $ R $, let $ K_{R, I} \subset \Gb(R) $ the subgroup of elements $ \gamma $ such that $ \gamma v  \in v +  I    V_{R} $ for all $ v \in V_{R} $.

\begin{lemma}   \label{KIorbits} Suppose that  $ R $ is a discrete valuation ring  and let $ \varpi $ be  a     uniformizer.    Then for any   $ v \in V_{R} $ and ideals $ I , J  \subset  R    $,
$$ K_{R, I } v + J V_{R}     =   \begin{cases}  
I_{v} V_{R} \setminus  (  \varpi    I_{v} 
   V_{R}   )      & \text{ if }  I =  R ,  \, I_{v}  \supsetneq 
   J \\  v    + ( J +  I \cdot I_{v} ) V_{R}  & \text{otherwise}.
\end{cases}  $$    
\end{lemma}

\begin{proof} First assume that $ J = 0 $. 
Then the  case $ I = R $ is Lemma \ref{GZorbits}, so we  also assume that  $  I $ is proper.     Let $ d_{v} \in I_{v} $ be a generator and  $  \kappa  \in \Gb( R  ) $ such that $ v = \kappa d_{v} e_{1} $.  Since $ K_{R, I } $ is normal in $  \mathbf{G}(R)$, 
$ K_{R, I} v = K_{R, I} ( \kappa d_{v} e_{1} ) =  \kappa d_{v}  (  K_{R, I}  e_{1} )  $   and  the claim is further reduced to the case $ v = e_{1} $.   But this holds since $ K_{R, I } $ contains matrices whose first column is $ [ 1 + a_{1},   a_{2}, \ldots, a_{2n}]^{t} $ for arbitrary $ a_{i} \in I  $. Now note that the cases for $ J \neq  0 $  follow easily 
 from the corresponding ones 
 for  $ J = 0  $.   
\end{proof}    

\begin{lemma}  \label{KMorbits}
Let $ M , N $ be positive integers such that  $ M | N $ and $ v \in V_{\ZZ} $. Let $ d_{v} \in \ZZ   $ denote a generator of $I_{v}$, $b = \gcd(Md_{v}, N )$ and $ S $ the set of primes   $ \ell   $      such that $ d_{v} \ZZ_{\ell} \supsetneq N \ZZ_{\ell}    $ and $ \ell \nmid M $.        
Then   $$  K_{M}v + N V_{\widehat{\ZZ}}   =  \prod_{  \ell \in S }    (       d_{v}  V_{\ZZ_{\ell}}  \setminus  \ell   d_{v}     V_{\ZZ_{\ell}}   )     \prod_{\ell  \notin S } ( v +  b    V_{\ZZ_{\ell}}   )       
$$ 
\end{lemma}    

\begin{proof}Since $ K_{M} = \prod_{\ell} K_{\ZZ_{\ell}, M\ZZ_{\ell}} $ and $ v + N V_{\widehat{\ZZ}}  = \prod_{\ell } (v  + N V_{\ZZ_{\ell}}) $, the claim    follows by Lemma  \ref{KIorbits}. 
\end{proof}

\begin{remark} If we write $ N = M d  $, the result above shows that $ (V_{\ZZ} / N V_{\ZZ})^{K_{M} / K_{N}  }  = d V_{\ZZ} / N V_{\ZZ} $.   
\end{remark}

\subsection{RIC functors}
The adelic distribution relations of Eisenstein classes 
are  most  conveniently   described as a morphism between two \emph{RIC functors} \cite[\S 2]{AESthesis} (cf.\ \cite[\S 2]{Anticyclo}). To make the note more self-contained, we briefly recall the terminology  and establish a basic result needed later on.

Fix for this subsection only a locally profinite group $ G $ and a non-empty collection $ \Upsilon $ of compact open subgroups. We assume that $ \Upsilon $ is closed under intersections, conjugation by elements of $G$ and for every $ K , L \in  \Upsilon $, there exists a $ K' \in \Upsilon  $ such that $ K ' \subset L $ and $ K' \triangleleft K  $. For such an $ \Upsilon $, we associate a category $   \mathcal{P}(G)    =    \mathcal{P}(G, \Upsilon)$ whose objects are elements of $\Upsilon$ and whose morphisms are given by  $ \mathrm{Hom}_{\mathcal{P}(G)} ( L, K ) =  \left \{ g \in G \, | \,  g^{-1} L g  \subset  K   \right \}  $ for $L$, $ K \in \Upsilon $. Elements of $\mathrm{Hom}_{\mathcal{P}(G)}(L, K) $ will be  written as either $ (L \xrightarrow{g} K) $ or $[g]_{L,K} $, and composition in $ \mathcal{P}(G) $ 
 is given by  $$ (L \xrightarrow{ g } K  ) \circ  (  L ' \xrightarrow {h } L ) =  (L' \xrightarrow{ h } L \xrightarrow{g} K )  =  ( L '  \xrightarrow{hg} K ) . $$ 
 If $  e   $ denotes the identity of $ G $, the inclusion $ (L\xrightarrow{e}K) $ will also denoted by $ \pr_{L,K} $. 

   \begin{definition}    
\label{RICdefinition}   Let $ R $ be a commutative ring with  identity.      An   \emph{RIC functor $ M $ on $ (G, \Upsilon ) $ valued in    $ R $-\text{Mod}}  is a pair of covariant functors $$  M ^ { * } :  \mathcal{P}(G, \Upsilon) ^ { \text{op} }  \to   R\text{-Mod}  \quad  \quad   \quad        M  _{  * } :  \mathcal{P}(G, \Upsilon)  \to  R\text{-Mod}  $$
satisfying   the   following   three   conditions:         
\begin{itemize}  
\item [(C1)] $ M^{*}(K) = M_{*}(K)$ for all $ K  \in   \Upsilon $. Denote  this common $ R $-module  by $ M(K) $.
\item [(C2)]For all $ K \in \Upsilon $ and $ g \in G $, $$ (gKg^{-1} \xrightarrow{g} K)^{*} =  ( K \xrightarrow{g^{-1} } gKg^{-1} )_{*}  \in  \mathrm{Hom}_{R\text{-Mod}}(M(K),M(gKg^{-1})) . $$
Here for a morphism $ \phi  \in   \mathcal{P} (   G   )  $,  we denote $ \phi_{*}  : = M_{*}(\phi) $ and $ \phi^{*}  : = M^{*}(\phi) $. 
\item [(C3)] $  [ \gamma ]_{K,K,*} : M(K) \to M(K) $ is the identity map  for all $ K \in \Upsilon $ and $ \gamma \in K $.
\end{itemize}    
We will denote the RIC functor above simply  as   $ M : \mathcal{P}(G, \Upsilon) \to R\text{-Mod}$.    We refer to the maps $ \phi^{*} $ (resp.,   $ \phi_{*} $)    in axiom (C2)   as the \emph{pullback}  (resp., \emph{pushforward}) induced by $ \phi $. If moreover the element of $ G $ underlying the morphism $ \phi $ is $ e $, we also refer to $ \phi^{*} = \pr^{*} $ (resp.,     $ \phi_{*} = \pr_{*} $) as a \emph{restriction} (resp., an  \emph{induction}). \end{definition} 

\begin{definition}  A \emph{morphism} $ \varphi :  M_{1} \to M_{2} $ between two RIC functors $M_{1}$, $M_{2} $ is a collection of morphisms $ \varphi(K) : M_{1}(K) \to M_{2}(K) $ for all $ K \in \Upsilon $ that together constitute a natural transformation $ M_{1, *} \to  M_{2, *}$ and a natural transformation $ M_{1}^{*} \to M_{2}^{*} $.  
If $ \varphi(K) $ is injective for all $ K $, we say that $M_{1} $ is a \emph{sub-functor} of $ M_{2} $.     
\end{definition}    
\begin{definition} Let $ M : \mathcal{P}(G, \Upsilon) \to R\text{-Mod} $ be an RIC functor. We say that  $ M $ is   
\begin{itemize}  
\item [(G)] $ \emph{Galois} $ if for all $ L , K \in \Upsilon $ such that  $ L \triangleleft K $,  we  have  $$ \pr_{L,K}^{*} : M (K) \xrightarrow{\sim} M(L)^{K/L} . $$
Here the action of $ \gamma \in K/L $ on $M(L)$ is via pullbacks induced by $ (L \xrightarrow { \gamma } L  )  $ and $ M(L)^{K/L}$ denotes the invariants of $M(L)$ under this action.      
   
\item [(Co)]  \emph{cohomological} if for all $ L , K \in  \Upsilon $  with   $  L   \subset K $,  
 $$ ( L  \xrightarrow {  e}  K )   _ { * }   \circ   (  L  \xrightarrow{ e  }  K    ) ^ { *  }  =  [K:L] \cdot   (  K  \xrightarrow {  e  }       K   )^{*} .         $$
That is,   the composition is multiplication by index $ [K:L] $ on  $ M(K) $.   
\item [(M)] \emph{Mackey}   
if for all $K, L, L' \in \Upsilon$ with $L ,L' \subset K$, we have a commutative diagram 
\begin{equation}   \label{Mackeydiagram}    
 \begin{tikzcd}[column   sep = large]
 \bigoplus_{\gamma} M(L_{\gamma}) \arrow[r,"{\sum \pr_{*}}"] & M(L) \\
 M(L') \arrow[r,"{\pr_{*}}",swap] \arrow[u,"{\bigoplus  [\gamma]^{*}}"] & M(K) \arrow[u,"{\pr^{*}}",swap] 
\end{tikzcd}
\end{equation}   
where the direct sum in the top left corner is over a fixed choice of coset representatives $\gamma \in K $ of  the  double   quotient     $ L \backslash K / L'$ and $L_{\gamma} =  L  \cap  \gamma L' \gamma^{-1}    \in    \Upsilon       $. The  condition  is then satisfied by  any such  choice   of     representatives of $ L \backslash  K /  L '  $.     
\end{itemize}   
If $ M $ satisfies both (M) and (Co), we will say that $M  $ is \emph{CoMack}.  If $ S $ is  an   $ R $-algebra, the mapping $ K \mapsto M(K) \otimes _{R}  S  $ is   an   $ S $-valued  RIC functor, which is cohomological or Mackey if $ M $  is   so.  
\end{definition}

\begin{remark} An RIC functor is referred to as a ``cohomology functor" in \cite{Anticyclo}. We prefer the former terminology, since the standard name for the axiom (Co) (\cite{Webb}) conflicts with the latter terminology.       
\end{remark}

\begin{definition}     \label{compatibleunderpulldefisub}           Let $ N : \mathcal{P}(G, \Upsilon) \to  R $-\text{Mod} be an  RIC functor and $ S \subset  \Upsilon $ a  non-empty 
subset.  Let $ \mathscr{G}  = \left \{ B_{K} \, | \, K \in S \right \} $ be a collection   of  $ R $-submodules $ B_{K}  \subset   N(K) $   indexed by $ K \in  S $.  We say that $ \mathscr{G} $  is \emph{compatible under pullbacks of $ N $}  if for all $ L$, $K \in S $ and $ g \in G $ satisfying $ g ^{-1} L g \subset  K  $,  the morphism $ [g]_{ L , K  } ^{* }  : N(K) \to N(L  ) $ sends $ B_{K}  $ to $ B_{L } $. We say that $ \mathscr{G} $ is \emph{compatible under  restrictions} if we only require the previous  condition for $ L $, $ K \in S $ with $ L \subset K $ and $ g $ the identity  element.   We similarly define compatibility under pushforwards and  inductions. 
\end{definition} 
By definition,  a family   $  \mathscr{G}  $  as above  constitutes a  sub-functor of $ N $ if $ S = \Upsilon $ and  $  \mathscr{G}    $ is compatible under both pullbacks and inductions (equivalently,  pushforwards and restrictions).    
 
\begin{definition}   \label{compatibleunderpulldefi} Let $ M  $, $  N :  \mathcal{P}(G, \Upsilon)   \to  R \text{-Mod} $ be   two RIC   functors, $  S   \subset \Upsilon $ be any  non-empty  subset  and  $ \mathscr{G} =  \left \{ B_{K}  \, | \, K \in S  \right \}   $ be  a family of submodules $ B_{K}  \subset      N(K) $  indexed by $ K \in  S $ that is  compatible under pullbacks of $ N$.  Let $  \mathscr{F}    = \left \{  \varphi _ {K  } : B_{K} \to M(K)  \, |  \,     K \in S   \right \} $ be a collection of $ R $-module  homorphisms indexed by $ S $.    
We say that   $ \mathscr{F}   $ is \emph{compatible under pullbacks of $M$} if for all $   L  $, $   K \in  S $  and $ g \in G $ satisfying   $   g^{-1}  L  g  \subset K $, we have $   [g] ^ {*} _ {L, K } \circ \varphi _ { K } =  \varphi_{ L }    \circ  [g]^{*}_{ L ,  K }    $   as   maps   $  B_{K}  \to  M(L) $. We similarly  define   compatibility of $ \mathscr{F} $    under  restrictions, pushforwards or inductions when $ \mathscr{G} $ has these properties  respectively.     
\end{definition}
Suppose $ B_{K} = N(K) $ for all $ K \in S $. Then a family $ \mathscr{F} $ as above  constitutes a \emph{morphism} $ \varphi : N \to M $ of RIC functors if   $ S =  \Upsilon     $ and  $ \mathscr{F} $ is compatible under pullbacks and inductions (equivalently, pushforwards and  restrictions). 
\begin{lemma}  Let $ M  $, $ N :  \mathcal{P}(G, \Upsilon )  \to  R \text{-Mod} $ be RIC  functors   such   that $ M $  is  Mackey and all restriction maps in $ M $ are injective.   Let    $ \mathscr{F} = \left \{  \varphi_{K}   : N ( K)  \to M ( K )   \, | \,  K \in  \Upsilon     \right \} $ be a family  of morphisms that is  compatible  under  pullbacks.      Then $ \mathscr{F} $ is compatible under    inductions and thus  constitutes  a morphism of RIC functors.   
\label{res+conj=ind}
\end{lemma}    

\begin{proof}    Let $ L , K \in \Upsilon $ with $ L \subset K $. Pick a $ K' \in \Upsilon $ such that $ K' \triangleleft K $, $ K' \subset L $. 
Since $  \pr_{ K'   ,  K  } ^ { * }  : M (  K  ) \to M(K')  $ is injective,    $ \pr_{L, K , * } \circ \varphi_{L}  =   \varphi _ { K } \circ \pr_{L , K, * } $
if and only if  \begin{equation}   \label{resconjind1}      \pr_{ K' , K } ^ { * } \circ \pr_{L,K,*} \circ \varphi_{L}  =  \pr _ { K' , K  }  ^  { *  } \circ  \varphi_{K}  \circ \pr_{L, K , * }   
\end{equation}    
as maps $ N (L)  \to  M(K' )    $.    Since $  \mathscr{F}     $ is  compatible under   restrictions,  
we have  $  \pr_{K', K  } ^ { * } \circ  \varphi_{K}  = \varphi_{K'} \circ  \pr_{K', K  } ^ { * }  $. As $ M $ is  Mackey,   (\ref{resconjind1})   is  equivalent  to  $$ \sum   \nolimits  _{ \gamma \in K / L }  [ \gamma  ] _ { K '  ,   L    }    ^ { * } \circ \varphi_{L} =  \sum   \nolimits     _{  \gamma \in K / L }   \varphi_{K' } \circ  [  \gamma ] ^ { * } _ { K ' ,  L }   $$
(see Lemma 2.1.11 in \cite{AESthesis} for details).    
But this clearly holds by pullback compatibility of $ \mathscr{F} $. 
\end{proof}

\subsection{Functors for $\mathrm{GSp}_{2n}$}
We now resume the notations introduced in \S \ref{Siegelreview} and \S \ref{moduli}. \label{functorsGSp2n}    
Fix for the rest of this   note    a     rational prime $ p $ and an integer $ c > 1 $ with $ (c,p) = 1 $. Denote by $ \Upsilon_{1}  $ the collection of all principal congruence subgroups $ K_{N} $ for $  N  \geq 3 $ satisfying $ (N,cp) = 1 $. Let  
$$   G  : = \Gb(\Ab_{f} ^{cp}  ) \times  \Gb(   \ZZ_ {cp} )    $$  where $ \ZZ_{cp} : = \prod_{\ell | cp } \ZZ_{\ell} $ and  $  \Ab    _{f}^{cp}  :    = \Ab_{f} / \ZZ_{cp} $.    We will also  denote $ V_{f}^{cp} :=  V_{\ZZ} \otimes \Ab_{f}^{cp} $, $ V_{\ZZ_{cp}} : =  V_{\ZZ} \otimes \ZZ_{cp} $, $ \mathbf{Z}(\QQ)^{cp} : =  \mathbf{Z}(\QQ) \cap G $  and $  I^{cp}    : = \left \{ a \in \ZZ \, | \, (a     , cp    ) = 1 \right \} $.    For $ a \in \QQ^{\times} $, we denote   by  $ z_{a} \in \mathbf{Z}(\QQ) \simeq  \QQ^{\times} $ the corresponding element of the center given by $2n$-copies of $ a $ on the diagonal.       
Let  $ \Upsilon $ be the collection of all compact open  subgroups  of $ G $ which     are contained in a $ G $-conjugate of a group in $ \Upsilon_{1} $   and which  are  of the form $ \Gb(\ZZ_{cp})   L $ for some $ L \subset  \Gb(\Ab_{f}^{cp}) $.  It is easily verified that both $ \Upsilon_{1} $, $\Upsilon $ satisfy the conditions of the previous subsection with respect to $ K_{1} $, $ G $ respectively.

Set   $ X : = V_{f} \setminus  \{    0 \}  $ and view elements of $ X \subset  V_{f} = \bigoplus_{i=1}^{2n} \Ab_{f} e_{i}    \simeq  \Ab_{f}^{2n} $ as column vectors.    There is a smooth right action $ X  \times G  \to X $ given  by $ (v, g) \mapsto g^{-1}  v   $, i.e., left matrix multiplication by inverse of $ g $.  Let $ \mathcal{S}^{cp}(X)  $ be the set   of   all $ \ZZ_{p}$-valued functions on  $ X $ of the form $ \phi_{cp} \otimes \phi^{cp} $ where $ \phi_{cp} = \ch(V_{\ZZ_{cp}}) $ is the  characteristic function of $ \ZZ_{cp} $    and $ \phi^{cp} $ is a locally constant compactly supported function on $   (V_{f}^{cp}) \setminus   \left \{  0   \right \}     $.  Then  $ \mathcal{S}^{cp}(X) $ is a smooth   left      $G$-representation via its action    on $ X $. We  let   $$   \label{RIC1}      \mathcal{S}      :   \mathcal{P} ( G ,  \Upsilon  )  \to  \ZZ_{p}\text{-Mod}  $$  denote  the RIC  functor associated with the representation $  \mathcal{S}  ^{cp}          (X) $. That is, for any $ K \in \Upsilon $, $ \mathcal{S}(K) =  \mathcal{S}^{cp}(X)^{K} $ with restriction, inductions and conjugations given in the obvious manner.

Fix now a non-negative    integer $ k $. Recall that for each $ K_{N} \in \Upsilon  $, we have   fixed a   choice $  \pi_{N} :  \mathcal{A}_{N}  \to \mathrm{Sh}(K_{N}) $ of a universal abelian scheme in \S \ref{moduli}. Let $ \mathscr{H}_{\ZZ_{p}} = \mathscr{H}_{K_{N}, \ZZ_{p}} $ be the corresponding sheaf of Tate modules.  For each $ N $ satisfying $ K_{N}  \in  \Upsilon_{1} $,     we denote $$ \mathcal{E}_{N, \QQ_{p}} =  \mathcal{E}^{k}_{N, \QQ_{p}}  : = \mathrm{H}^{2n-1}_{\et} \big(\mathrm{Sh}(K_{N}) , \mathrm{Sym}^{k}(\mathscr{H}_{\QQ_{p}})   (n)     \big )              .   $$
Given two such $ M , N  $ such that $ M | N $,  we have a restriction map $  \pr_{N,M}^{*} : 
   \mathcal{E}_{M,\QQ_{p}} \to \mathcal{E}_{N, \QQ_{p}} $ induced by unit adjunction for $ \pr_{N, M} = [\mathrm{id}]_{N,M} :  \mathrm{Sh}(K_{N}) \to \mathrm{Sh}(K_{M}) $ and the isomorphism $  \mathcal{A}_{N} \to  \pr_{N,M} ^{*} (\mathcal{A}_{M}) $ specified by (\ref{univisog})  for $ g  = \mathrm{id} $.   
Let $$  \widehat{\mathcal{E} }_{\QQ_{p}} = \widehat{\mathcal{E}}_{\QQ_{p}}  ^ { k  }     : = \varinjlim\nolimits _{N} \mathcal{E}_{N, \QQ_{p}} ^{k}   $$ where the limit is taken with respect to restriction maps for $ M | N  $. 
\begin{lemma} $ \widehat{\mathcal{E}}_{\QQ_{p}} $ is a smooth $ G $-representation.  
\end{lemma}  
\begin{proof} Given $ g \in  G $ and $ x \in \mathcal{E}_{M, \QQ_{p}} $ we can find $ z = z_{a} \in \mathbf{Z}(\QQ)^{cp} $ for some $ a \in I^{cp} $, an element $ h \in G $ and a multiple $ N \in I^{cp} $ of $ M $ 
such that $ g = zh $ and $ \eta (N/M) V_{\widehat{\ZZ}} \subset  V_{\widehat{\ZZ}}  \subset \eta  V_{\widehat{\ZZ}}  $ is satisfied for $ \eta \in \left \{ z^{-1} , h \right \} $. The action of $ \eta $ is described by the pullback $ [\eta]^{*}_{N,M} : \mathcal{E}_{M, \QQ_{p}} \to  \mathcal{E}_{N, \QQ_{p}} $ induced by adjunction for $ [\eta]_{N,M}  :   \mathrm{Sh}(K_{N})  \to  \mathrm{Sh}(K_{M})  $ and  the inverse of the isomorphism  between 
symmetric power  of  sheaves of Tate modules  induced by  the  (prime-to-$p$) isogeny $ {}^{\eta} \psi_{N, M } 
:\mathcal{A}_{N}  \to  { } ^ { \eta    }     \mathcal{A}_{M} $   in     (\ref{univisog}). Since  $ \psi_{z^{-1}}^{\mathrm{univ}} $ is just $[a] $, $ z_{a} = (z_{a^{-1}})^{-1} $ acts by $ a^{k} $. Then $ g \cdot x $ is defined to be the  element corresponding to $ a^{k} \cdot  [h]_{N,M}^{*}(x)  \in  \mathcal{E}_{N,\QQ_{p}} $.  This action is well-defined since the isogenies $ \psi_{\eta}^{\mathrm{univ}} $ for varying $ \eta $  satisfy an obvious cocycle condition. 
\end{proof}    Since $ \pr_{N,M} $ is Galois 
with Galois group $ K_{M} / K_{N} $,  we have $ \pr_{N,M}^{*} \circ \pr_{N,M,*} = \sum_{ \gamma \in K_{M} / K_{N} } [\gamma]_{N,M}^{*} $ and $ \pr_{N,M,*} \circ \pr_{N,M}^{*} = [K_{M}: K_{N}]  \cdot  \mathrm{id}  $. Using this, one deduces that the natural map from $ \mathcal{E}_{N, \QQ_{p}}$ to $ \widehat{\mathcal{E}}_{\QQ_{p}}  $ identifies the former with the $ K_{N}$-invariants of the latter and that $ \pr_{N,M}^{*} $ (resp., $ \pr_{N,M,*}$) are identified with inclusion (resp. $\sum_{\gamma \in K_{N} / K_{M}  }   \gamma  )$.    We let $$ \mathcal{E}_{\QQ_{p}}    : \mathcal{P}(G, \Upsilon) \to  \QQ_{p}\text{-Mod}$$  denote    the associated RIC functor, i.e., $ \mathcal{E}_{\QQ_{p}}(K) :   = (\widehat{\mathcal{E}}_{\QQ_{p}})^{K} $ with obvious choice for pullback and pushforward maps.   It is CoMack and Galois by construction.

Next we define an integral structure on $ \mathcal{E}_{\QQ_{p}} $.     Let  $ \Upsilon_{2}  \subset \Upsilon $ be the subset of all groups  that are contained in a congruence subgroup in $   \Upsilon_{1}   $.       For $ K \in \Upsilon_{2} $,   we let $ \pi_{K} : \mathcal{A}_{K} \to \mathrm{Sh}(K) $ be the abelian scheme given by pulling back $ \mathcal{A}_{M}    $ along the degeneration map $ \mathrm{Sh}(K) \to  \mathrm{Sh}(K_{M}) $ for some $ K_{M} \in \Upsilon_{1} $  that contains $K  $. Then $ \mathcal{A}_{K} $ is independent of the choice of $ M $, since $ K_{M_{1}} \cap K_{M_{2}} = K_{\mathrm{lcm}(M_{1},M_{2})} $. 
Let $ \mathscr{H}_{\ZZ_{p}} = \mathscr{H}_{K, \ZZ_{p}} $ the associated $ \ZZ_{p}$-sheaf on $ \mathrm{Sh}(K) $.        
If we pick $ K_{N} \in \Upsilon_{1} $ such that $ K_{N} \trianglelefteq K $,  the natural pullback map along $  \mathrm{Sh}(K_{N}) \to \mathrm{Sh}(K) $  identifies $ \mathrm{H}^{2n-1}  _ { \et }  \big  ( \mathrm{Sh}(K),   \mathrm{Sym}^{k} ( \mathscr{H}_{\QQ_{p}}  )      ( n  )      \big    ) $  with the  $ K/ K_{N} $-invariants of $ \mathcal{E}_{N, \QQ_{p}}  =    
\mathcal{E}_{\QQ_{p}}(K) $. Again this identification is independent of the choice  
of $ N $.  
Let $    \mathcal{E}  _ { \ZZ_{p}  }  
(K)  \subset \mathcal{E}_{\QQ_{p}}(K)  $ 
denote the  image of $    \mathrm{H}^{2n-1}_{\et} \big (\mathrm{Sh}(K) , \Gamma_{k}( \mathscr{H}_{ \ZZ_{p}})   ( n   )       \big ) $   under 
\begin{align*} 
\mathrm{H}_{\et}^{2n-1} \big  ( \mathrm{Sh}(K) , \Gamma_{k} ( \mathscr{H}_{ \ZZ_{p}} )  ( n )   \big)   & \xrightarrow{   - \otimes  \QQ_{p}      }  \mathrm{H}_{\et}^{2n-1}\big ( \mathrm{Sh}(K) , \Gamma_{k} (   \mathscr{H}_{   \ZZ_{p}  }    ) ( n )  \big ) \otimes \QQ_{p}  \\
& \xrightarrow[\sim]{ \,  \,  \,  \, \, \sigma_{k}      \,  \, \, \,  \,   }  \mathrm{H}^{2n-1}_{\et} \big (\mathrm{Sh}(K) ,   \mathrm{Sym}^{k} ( \mathscr{H}_{   \QQ_{p}} ) ( n )     \big )    \xrightarrow{\sim} \mathcal{E}_{\QQ_{p}}(K)     
\end{align*}  
where $ \sigma _{k} $ is the  isomorphism induced  by    the     map  
(\ref{canonicalisogaammaH}). 

\begin{lemma}   \label{conjcomp}  The family $ \{ \mathcal{E}_{\ZZ_{p}} (K) \, | \, K \in \Upsilon_{2}  \} 
$  is compatible under pullbacks and pushforwards   of $ \mathcal{E}_{\QQ_{p}}$.    
\end{lemma}

\begin{proof}  Let $ ( L  \xrightarrow{  g }   K ) \in \mathcal{P}(G, \Upsilon) $ be such that   $ L, K \in \Upsilon_{2} $.  We wish to show that $ [g]^{*}_{L, K} : \mathcal{E}_{\QQ_{p}}(K) \to  \mathcal{E}_{\QQ_{p}}(L) $ preserves the corresponding $\ZZ_{p}$-submodules.    Since $ 
\mathbf{Z}(\QQ)^{cp} $   acts   by invertible  scalars, 
we may assume wlog that that $ g $ is such that  $ V_{\widehat{\ZZ}}   \subset g V_{\widehat{\ZZ}} $.     Choose $ K_{M} , K_{N} \in  \Upsilon_{1} $ such that  $  K_{M}  \supset K $  and  $  L  \supset K_{N} $. Replacing $ N $ by a multiple, we may assume that   $ M | N $ and that $ g(M/N) V_{\widehat{\ZZ}} \subset V_{\widehat{\ZZ}} $. 
Recall that the isogeny  $ {}^{g} \psi_{N, M} :  \mathcal{A}_{N} \to  { } ^ { g   }    \mathcal{A}_{M} $ 
(\ref{univisog})  is given as the quotient of $ \mathcal{A}_{N} $ by the group $C_{\gamma}  \subset \mathcal{A}_{N}[N]  $  corresponding to the kernel of $ \gamma  : V_{\ZZ}/ N V_{\ZZ} \to V_{\ZZ} / N V_{\ZZ} $ defined by $ v \mapsto g^{-1} v $.  Since $ g^{-1} L g \subset K_{M} $, 
the right  action of $ L / K_{N}  $ on $ \mathcal{A}_{N} \xrightarrow{\sim}    \mathcal{A}_{L} \times  _ { \mathrm{Sh}(L) }   \mathrm{Sh}(K_{N}) $   preserves $ C_{\gamma }   $.  
Thus  
$  {}^{g}    \psi _ { N, M } $  descends to an isogeny $  {}^{g}\psi_{L,K} : \mathcal{A}_{L} \to   {}^{g}\mathcal{A}_{K} $ (where $ {}^{g} \mathcal{A}_{K} := [g]^{*} \mathcal{A}_{K} $)
giving an isomorphism 
\begin{equation}  \label {sheafiso} 
   \mathscr{H}_{\ZZ_{p}, L} \xrightarrow{\sim}   [g]^{*} \mathscr{H} _{\ZZ_{p}, K}    .          
\end{equation} 
Here $ [g]     : \mathrm{Sh}(L)  \to  \mathrm{Sh}(K)  $   denotes the map given by right multiplication by $ g $.  The pullback map    $  \mathrm{H}^{2n-1}_{\et}\big (\mathrm{Sh}(K), \Gamma_{k}(\mathscr{H}_{\ZZ_{p}})  ( n )      \big  )     \to  \mathrm{H}^{2n-1}_{\et} \big    ( \mathrm{Sh}(L),  \Gamma_{k}(\mathscr{H}_{\ZZ_{p}} ) ( n )  
 \big   )   $  defined using (\ref{sheafiso})   
 induces (after tensoring with $ \QQ_{p}$) 
 a map $ p_{g} :  \mathcal{E}_{\QQ_{p}}(K) \to \mathcal{E}_{\QQ_{p}}(L)   $  that   sends  $ \mathcal{E}_{\ZZ_{p}}(K) $ to $ \mathcal{E}_{\ZZ_{p}}(L)  $ by construction.    Since $ { }  ^{g} \psi_{N,M} $ is the pullback of ${}^{g} \psi_{L,K} $ along $ \mathrm{Sh}(K_{N}) \to \mathrm{Sh}(K)$,   
 $p_{g} $ is compatible with  $[g]^{*}_{N,M} : \mathcal{E}_{M,\QQ_{p}} \to \mathcal{E}_{N,\QQ_{p}}$ and therefore equal to the map $ [g]^{*}_{L,K} $ of $ \mathcal{E}_{\QQ_{p}} $.    A similar argument applies to the pushforward   $ [g]_{L,K}^{*}$.      
\end{proof}
For  $ K \in \Upsilon $ arbitrary, choose $ g \in G $ such that $ K' : =  g K g^{-1} \in \Upsilon_{2} $ and define $ \mathcal{E}_{\ZZ_{p}}(K)  :=  [g]_{K', K}^{*} ( \mathcal{E}_{\ZZ_{p}}(K')) $. This is independent of the choice of $ g $. Indeed if $ h \in G $ is such that $ K'' := hKh^{-1} \in \Upsilon_{2} $, $ [hg^{-1}]^{*} $ sends $ \mathcal{E}_{\ZZ_{p}}(K') $ to $ \mathcal{E}_{\ZZ_{p}}(K'') $ by Lemma \ref{conjcomp}. The same result implies that the family  $ \left \{ \mathcal{E}_{\ZZ_{p}}(K)  \, | \,  K \in \Upsilon \right \} $
is compatible under pullbacks and pushforwards.      
So the association $ K \mapsto \mathcal{E}_{\ZZ_{p}}(K) $ assembles into an RIC functor   
\begin{equation}  \label{RIC2}  
 \mathcal{E}  _ { \ZZ_{p}}    :  \mathcal{P}(G, \Upsilon)  \to  \ZZ_{p} \text{-Mod}  
 \end{equation}   
which is CoMack   but \emph{not}  necessarily  Galois.

\begin{remark} 
An alternative way to define $ \mathcal{E}_{\QQ_{p}} $ and its integral sub-lattice is to use the notion of equivariant sheaves associated to algebraic representations of $ \Gb $ along the lines of \cite[\S 4]{Anticyclo}.     It is also possible to incorporate the action of a  monoid $ \Sigma \subset \Gb(\Ab_{f}) $ as in  \emph{loc.\ cit.}  larger than the group   $G $ that allows one to define the action of certain Hecke correspondences at the prime $ p $ as well. 
\end{remark}

\subsection{The main result}      
\label{SchwartzShimurasection} 
Recall that for $ v \in V_{\widehat{\ZZ}}   
$ and $ N \geq 3 $,   we denote by  $ t_{v, N} : \Sh(K_{N}) \to \mathcal{A}_{N} $    the torsion section given by $ \eta_{N}^{\mathrm{univ}} ( v ) $.   For $ N  $  such  that $ K_{N} \in  \Upsilon_{1}   $ and $ v  \in V_{\widehat{\ZZ} }   \setminus   N  V_{\widehat{\ZZ} }    $,       
we denote  $$   \xi_{v,N } :=  \ch  ( v    + N  V_{\widehat{\ZZ}}  ) : X \to \ZZ_{p}  $$ 
the characteristic function of $ v + N V_{\widehat{\ZZ}}  \subset X $.          Note that $ v + N  V_{  \widehat{\ZZ}   } $    is  $ K_{N} $-stable and equals  the product 
$ V_{\ZZ_{cp}} ( v^{cp} +    N V _{\widehat{\ZZ}^{cp}}) $
where $ v ^{cp} $ denotes the image of $ v $  under the projection $ V_{f} \to V_{f}  ^{c   p}   $ and $ \widehat{\ZZ}^{cp} = \widehat{\ZZ}/\ZZ_{cp} $.     Thus   $ \xi_{v,N }   \in \mathcal{S}(K_{N})  $.    
By Theorem \ref{Kingsmain}, we have 
$   N^{k} { } _    {c} \mathrm{Eis}   ^ { k  }     _{\QQ_{p}}  (  t _{v,N} ) \in  \mathcal{E} _ { \ZZ_{p}     }     ( K_{N} )  $ for all    $ v \in V_{\widehat{\ZZ}}  \setminus  N  V_{\widehat{\ZZ}}   $.

\begin{theorem}    
\label{mainpararesult}       There exists a unique  morphism $   \varphi^{k}  :  \mathcal{S} 
\to   \mathcal{E} _{ \ZZ_{p}}   
$ of RIC  functors on $ \mathcal{P}(G   ,     \Upsilon)  $  
such that  for each $ K_{N} \in   \Upsilon_{1}  
$ and $  v \in V_{\widehat{\ZZ}} \setminus   N  V_{\widehat{\ZZ}}   $,  we have 
$  \varphi^{k}(K_{N}) 
 ( \xi_{v,    N } )  =   N ^ { k }       {}_{c}  \mathrm{Eis}^{k}_{\QQ_{p}  }  (  t _ {   v  ,    N    }   )   $. 
\end{theorem}      

\begin{proof}  We are going to construct this   morphism   in several   steps.  Since we exclusively work with $ \ZZ_{p} $-coefficients, we will denote $ \mathcal{E}_{\ZZ_{p}} $ simply as $ \mathcal{E}$.  
\\     

\noindent \textit{Step 1.} We first consider 
principal      levels.    For $ K_{M} \in   \Upsilon_{1}
$, let $ A_{M} \subset \mathcal{S}(K_{M}) $ denote the $ \ZZ_{p}$-span of $ \xi_{v, M} $ for $ v \in V_{\widehat{\ZZ}   }  \setminus 
M V_{\widehat{\ZZ}}   $ and let $ \varphi_{M} :  A_{M} \to \mathcal{E} (K_{M}) $ be the $ \ZZ_{p}$-linear map given by $ \xi_{v,M} \mapsto   M^{k}      { }  _   {  c  }    \mathrm{Eis}_{\QQ_{p}}^{k}(t_{v,M}) $.  This is well-defined since 
$ A_{M} $ is free over $ \xi_{v,M} $ for $ v $ running over  representatives  for $ (  V_{\ZZ}/ M V_{\ZZ}  ) \setminus  \left \{    0  \right \}  $.     Denote by $ \mathcal{S}_{1} , \mathcal{E}_{1}  $ the RIC functors  on   $   \mathcal{P} ( K_{1} , \Upsilon _{1} ) $ obtained by restricting the domain of $ \mathcal{S} $, $ \mathcal{E} $.  Let  $$ \mathscr{F} := \left \{ \varphi_{M} : A_{M} \to  \mathcal{E}(K_{M})  \, | \,   K_{M} \in  \Upsilon_{1}  \right \} 
$$ be the collection of all $ \varphi_{M} $.  
Clearly,   $ \left \{ A_{M} \, | \, K_{M}  \in\Upsilon _{1} \right  \} $ is compatible under pullbacks of $ \mathcal{S}_{1} $.  We claim that  $ \mathscr{F} $ is 
compatible under pullbacks of  $ \mathcal{E}_{1} $.  It suffices 
to check that for any $ ( K_{N}  \xrightarrow{ \kappa } K_{M}    ) \in  \mathcal{P}(K_{1},  \Upsilon_{1} )$ and $  v \in V_{\widehat{\ZZ} }   \setminus  M     V_{\widehat{\ZZ} } $,  the elements $  [\kappa] ^{*} \circ \varphi_{M} (\xi_{v,M} )  \in  \mathcal{E}(K_{N})  $ and $   \varphi_{N} \circ [\kappa] ^{*} (\xi_{v,   M} ) \in \mathcal{E}(K_{N}) $  are equal. Since  $ [\kappa]^{*} (\xi_{v,M}) =  \ch ( \kappa v  + M V _{\widehat{\ZZ}})  $  is the sum of $ \ch(\kappa v + M w  + N  V_{\widehat{\ZZ}}   )  $ for $ w \in V_{\ZZ}   /  d V_{\ZZ} $, we have        
\begin{align*}  [\kappa ] ^{*} \circ \varphi_{M} (\xi_{v,M} )   &  =  [\kappa ] ^{*}   \big (    M ^ { k } {}_{c}  
\mathrm{Eis}^{k}_{\QQ_{p}}( t_{v, M  }  ) \big )   \\ 
\varphi_{N} \circ [\kappa] ^{*} (\xi_{v,M} )  &  =   \sum  \nolimits   _{ w  \in V_{ \ZZ } / d V_{ \ZZ } }   N ^ { k   }           { } _ { c   } \mathrm{Eis} ^{k}_{\QQ_{p}}( t_{\kappa v +  M  w , N } ) . 
\end{align*}  
Now  Lemma  \ref{Eispullbackcompatibility}  applied  to (\ref{pullbackdiagram})  implies that $  [\kappa]^{*}  \big (  {}_{c}   
\mathrm{Eis}^{k}_{\QQ_{p}}( t_{v, M  }  ) \big )  =  {}_{c} \mathrm{Eis}^{k}_{\QQ_{p}}  \big ( [ \kappa]^{*} (t_{v,M})  \big    )  =   {}_{c} \mathrm{Eis}^{k}_{\QQ_{p}}(t_{ \kappa  d v , N } ) . $     But  Lemma  \ref{Eisnormcomp} (in conjunction with Lemma \ref{alphacnorm}) applied to the multiplication by $ d $ isogeny $ [d] : \mathcal{A}_{N} \to \mathcal{A}_{N} $ over $ \mathrm{Sh}(K_{N}) $ implies that $ {}_{c} \mathrm{Eis}^{k}_{\QQ_{p}}  (t_{ \kappa d  v , N } )  = \sum_{ w }  d^{k}   { } _ { c   }    \mathrm{Eis}^{k}_{\QQ_{p}}( t_{\kappa v + M w , N })    $. Thus  the two displayed equalities  above are themselves equal.    \\

\noindent \textit{Step 2.} Next we consider the action of center. Let $ \widehat{A} = \bigcup_{M}  A_{M} $ and let $ \widehat{\mathcal{E}}_{1} $ (resp., $\widehat{\mathcal{E}}$) be the inductive limit of $ \mathcal{E}(K) $ over all  $ K \in \Upsilon_{1} $ (resp., $ K \in \Upsilon$) with respect to restriction maps. By Step 1, we have an induced map $ \widehat{\varphi} : \widehat{A} \to \widehat{\mathcal{E}}_{1} $ of smooth $K_{1}$-representations.   As any element of $ \Upsilon $ contains an element of $  \Upsilon_{1} $, $ \widehat{\mathcal{E}}_{1} = \widehat{\mathcal{E}} $. 
So the target of $ \widehat{\varphi}$ is a $ G $-representation.  We show that $ \widehat{\varphi} $ extends   uniquely   to a map  \begin{equation}   \label{phihateis}                        \widehat{\varphi} :   \widehat{\mathcal{S}} \to   \widehat{\mathcal{E}} 
\end{equation} of $ ( \mathbf{Z}(\QQ)^{cp} K_{1} )     $-representations  as follows. First note that $ \widehat{A} $  is simply the subspace of all functions in $  \widehat{ \mathcal{S}  }  =        \mathcal{S}^{cp}(X) $ that are supported on $ V_{\widehat{\ZZ}}  $. Next recall that  $ \supp ( \phi ) $ for any   non-zero  $ \phi \in \widehat{\mathcal{S}}  $ is of the form $ V_{\ZZ_{cp}} Y $ for $ Y  \subset  
(V_{f}^{cp} ) \setminus \left \{ 0\right \} $. Since $ X = \bigcup _{  M \geq 1 } 1/ M \cdot  ( V_{\widehat{\ZZ}} \setminus \left \{ 0 \right \} ) $,  there exists a   positive integer $ a \in I^{cp} $ such that  $   a    \cdot \supp  ( \phi   ) \subset V_{\widehat{\ZZ}} $. So $ \phi = z_{a}^{-1} \cdot  \xi $ for some $ \xi \in  \widehat{A} $ and 
   the only possible extension is to set $$    
\widehat{\varphi} ( \phi )  :=  a^{-k}  \widehat{\varphi}( \xi) .   $$   
For this to be    well-defined, we must have   $ \widehat{\varphi} ( z_{a} \cdot \xi_{v,M} )$   equal to $    a^{k} \widehat{\varphi} ( \xi_{v,M}  ) $ for all $ a \in I^{cp} $,  $M$  satisfying $ K_{M} \in \Upsilon_{1} $ and   $ v \in V_{\widehat{\ZZ}   }  \setminus M V_{\widehat{\ZZ}} $. But this follows since $ z_{a} \cdot \xi_{v, M } = \xi_{av, aM} \in \mathcal{S}(K_{aM}) $ is mapped under $ \varphi_{aM} $ to $ (aM)^{k} {}_{c} \mathrm{Eis}^{k}_{\QQ_{p}} (t_{av,aM}) $  and this  class      coincides with  $ \pr_{aM, M}^{*} $ applied to $ a^{k} \varphi_{M} (\xi_{v,M}) = (aM)^{k}  { } _ { c }   \mathrm{Eis}^{k}_{\QQ_{p}}(t_{v,M}) $ by Lemma  \ref{Eispullbackcompatibility}.  
\\

\noindent \textit{Step 3.} We now  enlarge the domain        of each $ \varphi_{M}  $   to all of $ \mathcal{S}(K_{M})   $  for any    $ K_{M} \in   \Upsilon_{1}    $.  Let   $ B_{M}  \subset  \mathcal{S}(K_{M}) $ denote the $ \ZZ_{p}$-submodule of all finite sums  $ \sum_{i} z_{a_{i}} \cdot \xi_{i}$ where $ \xi_{i} \in A_{M}    $ and $ a_{i} \in \mathbf{Z}(\QQ)^{cp} $.    
For any $ \phi = \sum_{i} z_{a_{i}} \cdot \xi_{i} \in  B_{M} $, set $$ \varphi_{  M  }   ( \phi )   :     = \sum \nolimits _  {i }   a    _ {  i } ^ { k }      \,       \varphi_{M}( \xi_{i} ) . $$ 
Then $  \varphi_{    M   }        :      B_{M}  \to  \mathcal{E} (K_{M}) $ is  
well-defined (and uniquely determined)
by Step 2  and injectivity of restrictions of $ \mathcal{E}  $.  
We  claim that $ B_{M} = 
\mathcal{S}(K_{M}  )     $ for all $ M$. It suffices to show that $ \ch(C) \in B_{M }$ for $ C  \subset V_{\widehat{\ZZ}} $ any $K_{M}$-invariant compact open subset of the form $ V_{\ZZ_{cp}}Y $ for $ Y\subset V_{f}^{cp} \setminus \left \{ 0 \right \} $. For such $ C $, we can pick a $ N  = N_{C} \in I^{cp} $ a multiple of $ M $ such that $ C $ is a finite disjoint union of cosets in $ V_{\widehat{\ZZ}} / N V_{\widehat{\ZZ}} $.    Since $ C $ is also  $  K_{M} $-invariant, we can write $ C $ as a finite disjoint union of sets of the form $ K_{M}v + N V_{\widehat{\ZZ}} $ and we may  also assume $ v \in V_{\ZZ} $.   But  Lemma \ref{KMorbits} implies that $ \ch( K_{M} v + N V_{\widehat{\ZZ}}    ) $ can be written as a  difference of two  sums of functions of the form $ \ch ( aw + M a V_{\widehat{\ZZ}} ) $ for $ a \in I^{cp} $, $ w \notin M V_{\widehat{\ZZ} }   $. This  completes the step.  \\

\noindent   \textit{Step 4.}  We define     maps for levels in $ K \in \Upsilon_{2} $.  As in Step 3, we need to show that all elements of $ \widehat{\varphi}\big(\mathcal{S}(K)\big) $ lift to $ \mathcal{E}(K) $. 
Fix   any   $  K_{N} \in  \Upsilon_{1} $  such that $ K_{N}  \subset K  $.  Recall that $ K \subset K_{1}  $  acts on (the left of) $ V_{\ZZ} / N V_{\ZZ} $.  For any $v  \in V_{\widehat{\ZZ}} \setminus N V_{\widehat{\ZZ}} $, let $ K_{v}   \subset K $ denote the stabilizer of $ v  + N V_{\widehat{\ZZ}} \in V_{\ZZ}  / N V_{\ZZ} $.
Since $  \pr_ { K _ { N }    ,  K_{v}} $ is Galois 
and $ v + N V_{\widehat{\ZZ}} $ is $  K_{v} /K_{N} $-invariant,  the  section 
$ t_{v, N }  =  \eta_{N}^{\mathrm{univ}}(v)  $ descends to a $ N$-torsion section $ t_{v, K_{v}} :  \mathrm{Sh}(K_{v}) \to \mathcal{A}_{K_{v}} $. Thus  for any $ \gamma \in K  $, we  have $  [\gamma]_{K_{v}, K_{N} } ^{*}  (t_{v, K_{v}} )  =   t_{ \gamma v , N } $.   Now 
note  that   $ \ch ( K v + N V_{\widehat{\ZZ}} ) = 
\sum_{ \gamma \in K / K_{v} } \xi_{\gamma v, N   }    $ is an element of $ \mathcal{S}(K) $.
We define $$ \varphi_{K} \big ( \ch(K  v     + N V_{\widehat{\ZZ}}  )   \big   )      :  =
        \pr_{K_{v} , K  ,  *  }     \big  (   N^{k}        {}_{c}\mathrm{Eis}^{k}_{\QQ_{p}}(t_{v, K_{v}})  \big    )   $$
which   belongs  to 
$ \mathcal{E}(K)  $ by Theorem \ref{Kingsmain}.    
That 
this  agrees with $
\sum _ { \gamma \in K /   K _ { v  }     }  \widehat{\varphi} (  \xi_{\gamma  v , N }  )   $ in $ \widehat{\mathcal{E}}   $ 
follows since
$$ \pr_{K, K_{N} }^{*}  \circ \pr_{K_{v}, K, *} \big (N^{k} {} _{c}    \mathrm{Eis}_{\QQ_{p}}^{k}(t_{v,K}) \big )  = \sum  \nolimits  _{  \gamma \in  K /   K _ { v  }     }   [\gamma]_{K_{v}, K_{N} }^{*}   \big (      N^{k}    {}_{c} \mathrm{Eis}_{\QQ_{p}}^{k}(t_{v, K_{v}})   \big ) $$ by   the     Mackey axiom  
and  since     $ 
    \varphi_{N} ( \xi_{\gamma v, N } ) 
    =  [\gamma]_{K_{v} , K_{N} } ^{*}   \big (   N^{k}      {}_{c} \mathrm{Eis}_{\QQ_{p}}^{k}(t_{ v ,  K_{v}    }  ) \big )  $ for any $ \gamma \in  K    $  
by Lemma \ref{Eispullbackcompatibility}.  

As in Step 3, we can use the above  to define $ \varphi_{K} $ for any finite linear combination of functions above scaled by elements of $ \mathbf{Z}(\QQ)^{cp}    $. So it only remains to argue that all elements of $ \mathcal{S}(K) $ are of this form. Again,  it suffices to show this for characteristic functions $ \ch (C) \in \mathcal{S}(K) $ for some $ C \subset V_{\widehat{\ZZ}}  $. But this follows since we can find a sufficiently large $ N=  N_{C} \in I^{cp} $ such that $  K_{N} \subset K $ and $ C $ is a finite disjoint union of sets of the form $ K v + N V_{\widehat{\ZZ}} $.  
 \\

\noindent  \textit{Step 5.} The final step is to show that  the family 
$ \left \{ \varphi_{K} \, | \,  K  \in \Upsilon_{2 }   \right \}$ 
extends  uniquely  to a   morphism of functors on $ \mathcal{P}(G, \Upsilon) $. To this end,  it  suffices  to   establish     $  \widehat{\varphi} $ (\ref{phihateis})  is a map of $ G $-representations. Indeed, any $ K \in \Upsilon $ is such that $ g K g^{-1}  \in  \Upsilon_{2} $ 
for some $ g \in G $ and we can define $ \varphi_{K}(\xi) $ as  $   [ g ] _ {  g K g ^{-1}  , K , *  }  \circ    \varphi_{gKg^{-1}}(g \cdot \xi) $. The resulting family of homomorphisms indexed by $ \Upsilon $      would then be  compatible under pullbacks by $  G $-equivariance 
and  Lemma \ref{res+conj=ind} would  give the  desired claim.

Fix $ g \in \Gb(\Ab_{f}) $ and $ \xi \in \widehat{\mathcal{S}} $. We wish to show that $ \widehat{\varphi}(g \cdot \xi) =  g \cdot   \widehat{\varphi}(\xi)   $.     Recall that  $ \widehat{\varphi} $ was shown to equivariant with respect to $ \mathbf{Z}(\QQ) ^ {  c p   }   $ in Step 2. Since any $ \xi \in \widehat{\mathcal{S}} $ is $ K_{M}  $-invariant for some $ K_{M }   \in \Upsilon_{1}    $, 
it   suffices  to restrict to the case  where    $ \xi = \xi_{v,M} $ for some $ v $, $ M $  and $ g $   satisfies    
  $ V_{\widehat{\ZZ}}  \subset g V_{\widehat{\ZZ} }   $.  Let  $ N \in  I^{cp}     $ be a
multiple  of  $ M $  such that
$ g ( N/M) V_{\widehat{\ZZ}}   \subset   V_{\widehat{\ZZ}}   $, $  {}^{g} \psi  
: \mathcal{A}_{N}  \to   {  }  ^ { g  }   \mathcal{A}_{M}  $ denote the   isogeny in (\ref{univisog})  and   $    {}^{g}t_{v,M} :  \mathrm{Sh}(K_{N}) \to  {}^{g} \mathcal{A}_{M} $ denote the torsion section that is obtained as the pullback of $ t_{v,M}  :  \mathrm{Sh}(K_{M}) \to  \mathcal{A}_{M}  $ under 
$  [g]_{N, M }  : \mathrm{Sh}(K_{N})  \to  
  \mathrm{Sh}(K_{M})  $. 
\begin{equation} \label{twisiso} 
\begin{tikzcd}[sep = large]
\mathcal{A}_{N} \arrow[rd, "\pi_{N}"'] \arrow[r, "{}^{g} \psi"] & {}^{g}\mathcal{A}_{M} \arrow[d] \arrow[r]   & \mathcal{A}_{M} \arrow[d, "\pi_{M}"']  \\
& \mathrm{Sh}(K_{N}) \arrow[r, "{[g]_{N,M}}"'] \arrow[u, "{{}^{g}t_{v,M}}"', dashed, bend right, shift right=2] & \mathrm{Sh}(K_{M}) \arrow[u, "{t_{v,M}}"', bend right, shift right=2]
\end{tikzcd}
\end{equation}
Then the  universal torsion sections of $ \mathcal{A}_{N} $ that map to $  { } ^ { g }  t_{v,M}  $  under  $  { } ^ { g }     \psi       $  are $ t_ {  gdv + w, N} $  where $ d : = N / M $ and   $ w \in g N V_{\widehat{\ZZ}} $ runs  over representatives of $ g  (  N V_{\widehat{\ZZ}}   )     / N  V_{\widehat{\ZZ} }   \subset  V _ {   \ZZ } / N V_{ \ZZ    }   $.
Now $  [ g ]_{N, M } ^{ * }  \cdot \xi_{v,M} 
=  z_{d}^{-1} \cdot   \ch (  gd v +  g N V_{\widehat{\ZZ}}) $ and the right hand side expands as $ \sum \nolimits _ { w  
} z_{d}^{-1} \cdot \ch (   g  d     v + w + N V_{\widehat{\ZZ}} ) $ with $ w $ running over $ g (   N V_{\widehat{\ZZ}}   )  / N V_{\widehat{\ZZ}}  $. So we  need to show that 
$$   [g]_{N,M}^{*}   \left (  M^{k}  {}_{c}   \mathrm{Eis}   ^  { k  } _ { \QQ_{p} }  ( t_{v, M } )     \right  )          = \sum   \nolimits  _{ w } \mathrm{Sym}^{k} ( {}^{g}  \psi  )_{*} \left  (  d^{-k}  N ^{k} \,   { } _ {c}  \mathrm{Eis}^{k}_{\QQ_{p}}(t_{ gdv   + w , N }    )      \right     ) .   $$   But this follows  by Lemma \ref{Eispullbackcompatibility} and Lemma \ref{Eisnormcomp} as before. 
\end{proof}    

By Step 3 of the proof, the image of $ \varphi^{k}(K_{N}) $ is the $ \ZZ_{p}$-linear span of Eisenstein classes along non-zero universal $N $-torsion sections. An interesting implication of this is the following. 
\begin{corollary} For any $ N \geq 3 $ prime to $ cp $,  the $ \ZZ_{p}$-submodule of  $ \, \mathrm{H}^{2n-1}_{\et }   \big  ( \mathrm{Sh}(K_{N}) ,  \mathrm{Sym}^{k}(\mathscr{H}_{\QQ_{p}})(n)  \big  )  $ spanned by $ {}_{c} \mathrm{Eis}^{(k)}_{\QQ_{p}}(t_{v,N}) $ for  non-zero $ v \in V_{\ZZ} / N V_{\ZZ} 
$ is stable under the natural action of Hecke correspondences $ [K_{N} g K_{N}] $ for any $ g \in G $.    
\end{corollary}

\begin{remark} Note that our parametrization result is only made for the \emph{image} of the integral Eisenstein classes \cite[Definition 6.4.3]{Kings15} in the $ \QQ_{p}$-cohomology, and not for the integral classes themselves.   For $ n = 1 $ and $ c $ satisfying $ (c, 6p) = 1 $, an alternative  parametrization for these integral classes can be obtained as follows.   Since the units of the structure sheaf of modular curves have Galois descent, it is straightforward to define Siegel units associated with arbitrary Schwartz functions. One can then exploit \cite[Theorem 4.7.1]{Kings14} and the compatiblity of Ohta's twisting morphism with Kings' moment maps (see \cite[Theorem 4.5.1(2)]{KLZ}) to define `Eisenstein classes' integrally for any Schwartz function of the form specified in \cite[Definition 9.1.3]{LSZ}. These classes then agree with the integral Eisenstein classes defined by Kings (up to a normalization factor)  when the Schwartz function corresponds to a genuine $N$-torsion section by the results of \cite{Kings14}.    The adelic distribution relations of these classes then immediately follow from those of the Siegel units.    We are grateful to David Loeffler for sharing this observation.    
\end{remark} 
\bibliographystyle{amsalpha}    
\bibliography{refs} 
\Addresses    
\end{document}